\documentclass[11pt,reqno]{amsart}
\usepackage{amsmath, amssymb, graphicx, mathrsfs, hyperref}
\usepackage{amsthm}
\usepackage{palatino}
\usepackage{autonum}
\usepackage[makeroom]{cancel}
\usepackage{enumerate} 
\usepackage[usenames,dvipsnames,svgnames]{xcolor}
\usepackage[top=1.2in, bottom=1in, left=1in, right=1in]{geometry} 
\allowdisplaybreaks

\renewcommand{\(}{\left\(}
\renewcommand{\)}{\right\)}
\renewcommand{\[}{\left\[}

\renewcommand{\]}{\right\]}

\numberwithin{equation}{section}
\theoremstyle{plain}
\newtheorem{theorem}{Theorem}[section]
\newtheorem{lemma}[theorem]{Lemma}
\newtheorem{remark}[]{Remark}
\newtheorem*{remark*}{Remark}

\newtheorem{corollary}[theorem]{Corollary}

\theoremstyle{definition}

\makeatletter
\def\proof{\@ifnextchar[{\@oproof}{\@nproof}}
\def\@oproof[#1][#2]{\trivlist\item[\hskip\labelsep\textit{#2 Proof of\
		#1.}~]\ignorespaces}
\def\@nproof{\trivlist\item[\hskip\labelsep\textit{Proof.}~]\ignorespaces}

\makeatother

\begin{document}
	\title[Higher Order Dualities over Global Function Fields and Weighted M\"{o}bius Sums over $\mathbb{F}_q{[T]}$]{Higher Order Dualities over Global Function Fields and Weighted M\"{o}bius Sums over $\mathbb{F}_q{[T]}$} 
	
	\author{Prassanna Nand Jha and Jagannath Sahoo}

	\address{Department of Mathematics, Indian Institute of Technology Gandhinagar, Palaj, Gandhinagar 382355, Gujarat, India} 
	\email{prassanna.jha@iitgn.ac.in}
	\address{Department of Mathematics, Indraprastha Institute of Information Technology Delhi, New Delhi, 110020, India}
	\email{jagannaths@iiitd.ac.in, jagannath.sahoo@alumni.iitgn.ac.in}

	\begin{abstract}
		 Alladi's duality identities (1977) provide a fundamental relation between the smallest and the $k$-th largest prime factors of integers. In this paper, we establish Alladi-type duality identities over global function fields, extending a result of Duan, Wang, and Yi. We also prove the asymptotic vanishing of a function field analogue of the weighted M\"{o}bius sum $\sum \mu(n)\omega(n)/n$, where $\omega(n)$ denotes the number of distinct prime divisors of $n$. By estimating $\Psi_2(n,m)$, the number of monics of degree $n$ whose second largest prime divisor degree is at most $m$, we further show that this vanishing persists when the sum is restricted to monics with a unique prime divisor of smallest degree belonging to a set of primes with natural density. As a corollary, we show that the unrestricted sum decomposes into infinitely many sub-series, each vanishing asymptotically.
		
	\end{abstract}
	
	\maketitle

	\section{Introduction}\label{sec:intro}
	
	Sums involving the M\"{o}bius function play a fundamental role in analytic number theory and their study is a central theme of the subject. Its close connection with the distribution of primes is revealed by Landau's theorem, which states that the prime number theorem is equivalent to the identity
	\begin{align}\label{PNTeq}
		\sum_{n\ge2}\frac{\mu(n)}{n}=-1,
	\end{align}
	where $\mu$ denotes the M\"obius function. In 1977, Alladi \cite{alladi} obtained a refinement of this identity by imposing conditions on the smallest prime factor. More precisely, if $p_1(n)$ denotes the smallest prime factor of $n$, then for positive integers $a$ and $m$ with $(a,m)=1$, he proved that
	\begin{align}\label{alladi-main-eq}
		\sum_{\substack{n\geq 2\\ p_1(n)\equiv a \bmod m}}
		\frac{\mu(n)}{n}
		=
		-\frac{1}{\varphi(m)}.
	\end{align}
	Thus, the restricted M\"{o}bius sum detects the density of the reduced residue class containing the smallest prime factor.
	In the same spirit, Alladi and Johnson \cite{alladi-johnson} later studied these M\"{o}bius sums weighted by $\omega(n)$, the number of distinct prime divisors of $n$. They proved that
	\begin{align}\label{alladi-johnson-eqs}
		\sum_{\substack{n\ge2}}
		\frac{\mu(n)\omega(n)}{n}=0, \qquad \text{and} \quad \sum_{\substack{n\ge2\\ p_1(n)\equiv a \bmod m }}
		\frac{\mu(n)\omega(n)}{n}=0.
	\end{align}
	These M\"{o}bius sums have since been generalized in various directions. We give an overview of the surrounding literature in Section \ref{sec:motivation}.

	The fundamental ingredient underlying these results are duality identities discovered by Alladi \cite{alladi}, which allow information about the $k$-th largest prime factor to be recovered from weighted M\"{o}bius divisor sums involving the smallest prime factors. More precisely, if $P_k(n)$ and $p_k(n)$ denote the $k$-th largest and $k$-th smallest prime factors of $n$, respectively, then Alladi proved that for every arithmetic function $f$ satisfying $f(1)=0$,
	\begin{align}\label{alladi-duality}
		\sum_{d\mid n}
		\mu(d)
		\binom{\omega(d)-1}{k-1}
		f(p_1(d))
		=
		(-1)^k
		f(P_k(n)).
	\end{align}
	
	Function fields provide a particularly attractive setting for investigating these identities. Besides their close analogy with the integers, the arithmetic of $\mathbb{F}_q[T]$ often admits sharper counting results for prime polynomials, allowing combinatorial arguments to be carried out more explicitly while retaining many of the essential features of the classical theory. In recent work, Duan, Wang and Yi \cite{duan-wang-yi} established a first order Alladi duality over global function fields and applied it to obtain a function field analogue of \eqref{alladi-main-eq}. Their work naturally raises the question of whether higher-order duality identities also exist over global function fields and whether they lead to weighted M\"{o}bius sum identities analogous to \eqref{alladi-johnson-eqs}.
	
	The principal objective of this paper is to answer these questions. We establish higher-order Alladi duality identities over rational and global function fields, in Theorems \ref{thm:duality in ff} and \ref{thm:global_duality} respectively, extending the first-order result of \cite{duan-wang-yi}. The proofs rely on new combinatorial identities for weighted M\"{o}bius sums over divisors and yield explicit higher-order analogues of the classical duality formula.
	
	Our second objective is to demonstrate the effectiveness of these higher-order identities. To this end, we investigate weighted M\"obius sums over $\mathbb{F}_q[T]$ analogous to those studied by Alladi and Johnson in \cite{alladi-johnson}. A key ingredient is a new upper bound given in Theorem \ref{thm:Psi2} for the second-order smooth polynomial counting function
	\begin{align}
		\Psi_2(n,m)
		=
		\#\left\{
		A\in\mathbb{F}_q[T]:
		A\ \text{monic},
		\deg A=n,
		\Delta_2(A)\le m
		\right\},
	\end{align}
	where $\Delta_2(A)$ denotes the second largest element of the set $\mathcal{N}(A)=\{\deg P : P \text{ prime}, P\mid A\}$. This estimate appears to be of independent interest and forms the principal analytic input in our arguments.
	In Theorem \ref{thm:landau_two_FF}, we establish the asymptotic vanishing of M\"obius sums weighted by $\Omega(A)=\#\mathcal{N}(A)$, analogous to the first identity in \eqref{alladi-johnson-eqs}. Then, combining it with the second-order duality, we show in Theorem \ref{thm:mainthm AJ} that the asymptotic vanishing in Theorem \ref{thm:landau_two_FF} persists even when the sum is restricted over monics with a unique prime divisor of smallest degree belonging to a set of primes with natural density. As a consequence, we show in Corollary \ref{cor:to main thm in ff} that the unrestricted weighted M\"obius sum decomposes into infinitely many asymptotically vanishing subseries.
	
	The paper is organized as follows. In Section \ref{sec:results}, we state our main results over the rational function field $\mathbb{F}_q[T]$, while Section \ref{sec:global} extends the higher-order duality identities to arbitrary global function fields. Section \ref{sec:motivation} discusses related work and places the present results in the context of the existing literature. Section \ref{prelems} collects the necessary preliminary results and the remaining sections contain the proofs of the higher-order duality identities, the estimate for $\Psi_2(n,m)$, and the applications to weighted M\"obius sums.

	\section{Statement of results}\label{sec:results}
	Let $q$ be a prime power and $\mathbb{F}_q$ be the finite field with $q$ elements. Let $\mathbb{F}_q[T]$ denote the ring of polynomials over $\mathbb{F}_q$ and $\mathbb{F}_q(T)$ denote its quotient field. The field $\mathbb{F}_q(T)$ is called a rational function field and any finite extension of $\mathbb{F}_q(T)$ is called a global function field.
	The rational function field $\mathbb{F}_q(T)$ possesses many properties similar to the field $\mathbb{Q}$. As a result, many classical number theoretic results on $\mathbb{Z}$ have analogous translations over the ring $\mathbb{F}_q[T]$.
	
	We say that a monic polynomial over $\mathbb{F}_q$ is prime if it is irreducible in the ring $\mathbb{F}_q[T]$. Let $\mathcal{P}$ denote the set of all primes in $\mathbb{F}_q[T]$ and let $\mathcal{S}$ be an arbitrary subset of $\mathcal{P}$. For integers $n\geq 0$, define $\pi_\mathcal{S}(n)$ to be the cardinality of the set $\{P\in \mathcal{S}:\deg P=n\}$, where $\deg P$ denotes the degree of the polynomial $P$. The natural density $\delta(\mathcal{S})$ of $\mathcal{S}$ is defined as
	\begin{equation}
		\delta(\mathcal{S}):=\lim_{n\to\infty}\frac{\pi_{\mathcal{S}}(n)}{\pi_\mathcal{P}(n)},
	\end{equation}
	provided the limit exists. For a monic polynomial $A\in \mathbb{F}_q[T]$, we define the M\"{o}bius function as
	\begin{equation}
		\mu(A):=
		\begin{cases}
			1 & \text{if } A=1, \\
			(-1)^k & \text{if } A \text{ is a product of } k \text{ distinct primes}, \\
			0 & \text{otherwise}.
		\end{cases}
	\end{equation}
	Let $\mathcal{N}(A)=\{\deg P:\text{ prime }P\mid A\}$ and $\Omega(A)=\#\mathcal{N}(A)$, the cardinality of $\mathcal{N}(A)$. Let $\Delta_k(A)$ and $\delta_k(A)$ denote the $k$-th largest and the $k$-th smallest elements of $\mathcal{N}(A)$ respectively. Define $\Delta_k(A)=\delta_k(A)=0$ for $k>\Omega(A)$. Also, define
	\begin{align}
		&P_{\min}(A):=\{\text{prime }P\mid A:\deg P=\delta_1(A)\}, \\
		&\mathcal{D}(\mathcal{S}):=\{\text{monic }A\in\mathbb{F}_q[T]:P_{\min}(A)\text{ has only one element }P\text{ and }P\in\mathcal{S}\}, \text{ and}\\
		&Q_\mathcal{S}^{(k)}(A):=\#\{P\in\mathcal{S}:P\mid A,\deg P=\Delta_k(A)\}.\label{important_definitions}
	\end{align}
	Throughout this paper, let $\sideset{}{'}\sum$ denote a sum over monics. Duan, Wang and Yi \cite[Lemma 3.1]{duan-wang-yi} showed that for any monic $A\in\mathbb{F}_q[T]$ and $f:\mathbb{N}\to\mathbb{C}$ with $f(0)=0$, we have
	\begin{equation}\label{duan-duality}
		\sideset{}{'}\sum_{B\in\mathcal{D}(\mathcal{S}),\hspace{1mm}B\mid A}\mu(B)f(\delta_1(B))=-Q_\mathcal{S}^{(1)}(A)f(\Delta_1(A)).
	\end{equation}
	Our first main result generalizes this duality to higher orders.
	\begin{theorem}\label{thm:duality in ff}
		Let $A\in\mathbb{F}_q[T]$ be monic. If $f:\mathbb{N}\to\mathbb{C}$ is an arithmetic function with $f(0)=0$, then for any positive integer $k$,
		\begin{equation}\label{gen duality in ff}
			\sideset{}{'}\sum_{B\in\mathcal{D}(\mathcal{S}),\hspace{1mm}B\mid A}\mu(B){\Omega(B)-1\choose k-1}f(\delta_1(B))=(-1)^kQ_\mathcal{S}^{(k)}(A)f(\Delta_k(A)),
		\end{equation}
		where
		$\Omega(A)=\#\{\deg P:\text{ prime }P\mid A\}$.
	\end{theorem}
	\begin{remark}
		The duality result \eqref{duan-duality} by Duan, Wang and Yi \cite{duan-wang-yi} is stated more generally for primes in global function fields in their paper. Our result can also be formulated in the context of global function fields. Since this statement involves setting up some notation, we state it in Section \ref{sec:global}.
	\end{remark}
	
	Next, to aid in proving the subsequent theorems, we obtain the following bound for 
	\begin{align}
		\Psi_2(n,m):=\#\{A \in \mathbb{F}_q[T] : A \text{ is monic}, \deg A = n, \Delta_2(A)\leq m\}.
	\end{align}
	\begin{theorem}\label{thm:Psi2}
		Let $n$ be a positive integer and $m>0$ be such that $m=o(n)$. Then, we have
		$$\Psi_2(n,m)\ll_\epsilon \frac{q^n}{n^\epsilon}+ q^n \frac{m}{n},$$
		for every $0<\epsilon<1$.
	\end{theorem}
	
	We now shift our attention to weighted M\"{o}bius sums associated with the number of distinct prime degrees of a polynomial.
	
	\begin{theorem}\label{thm:landau_two_FF}
		We have that
		\begin{align}\label{landau_two_FF}
			\lim_{x\to\infty}\sideset{}{'}\sum_{1\leq\deg A\leq x}\frac{\mu(A)\Omega(A)}{q^{\deg A}}=0,
		\end{align}
		where
		$\Omega(A)=\#\{\deg P:\text{ prime }P\mid A\}$ and $\mu$ is the M\"{o}bius function.
	\end{theorem}

	Let $\mathcal{S}$ be an arbitrary subset of primes in $\mathbb{F}_q[T]$ having some natural density. As an application of Theorem \ref{thm:duality in ff} for $k=2$, we obtain that \eqref{landau_two_FF} vanishes asymptotically even when the sum is restricted to the elements whose smallest-degreed prime factor is unique and lies in $\mathcal{S}$.
	\begin{theorem}\label{thm:mainthm AJ}
		Let $\mathcal{S}$ be any arbitrary subset of the set of primes in $\mathbb{F}_q[T]$ with natural density $\delta(\mathcal{S})$. Then
		\begin{equation}
			\lim_{x\to\infty}\sideset{}{'}\sum_{\substack{1\leq\deg A\leq x\\A\in\mathcal{D}(\mathcal{S})}}\frac{\mu(A)\Omega(A)}{q^{\deg A}}=0,
		\end{equation}
		where
		$\Omega(A)=\#\{\deg P:\text{ prime }P\mid A\}$ and $\mu$ is the M\"{o}bius function.
	\end{theorem}
	
	\begin{remark}
		We expect a generalization of Theorem \ref{thm:mainthm AJ} to hold for global function fields as well. Proving this would require uniform estimates analogous to \eqref{eq : Psi1 in ff} for global function fields, which do not seem to be available in the existing literature. We relegate this study to forthcoming work.
	\end{remark}
	
	As a corollary of Theorem \ref{thm:mainthm AJ}, we obtain a refinement of Theorem \ref{thm:landau_two_FF}. We show that the sum in Theorem \ref{thm:landau_two_FF} admits a decomposition into an infinite family of subseries, all of which vanish asymptotically.
	\begin{corollary}\label{cor:to main thm in ff}
		For any positive integer $n$, we have
		\begin{equation}
			\lim_{x\to\infty}\sideset{}{'}\sum_{\substack{1\leq\deg A\leq x\\\delta_1(A)=n}}\frac{\mu(A)\Omega(A)}{q^{\deg A}}=0.
		\end{equation}
	\end{corollary}

	\section{Higher order Dualities over Global Function Fields}\label{sec:global}
	
	In this section, we formulate Theorem \ref{thm:duality in ff} in the setting of global function fields. Let $K$ be a global function field with constant field $\mathbb{F}_{q}$. A prime divisor (or simply a prime) in $K$ is a discrete valuation ring $R_{P}$ with maximal ideal $P$ such that $\mathbb{F}_{q}\subset R_{P}$ and the quotient field of $R_{P}$ is $K$. The associated residue field at $P$ is defined as $\kappa_{P}:=R_{P}/P$, which is a finite extension of $\mathbb{F}_{q}$. We define 
	\begin{align}
		\deg P = [\kappa_P : \mathbb{F}_q],
	\end{align}
	the degree of the field extension.
	The divisor group $\mathfrak{D}$ of $K$ is the free abelian group generated by $\mathcal{P}$, the set of primes in $K$, defined as
	\begin{align}
		\mathfrak{D}:=\left\{\sum_{P\in\mathcal{P}}a_{P}P : a_P \in \mathbb{Z} \text{ and } a_P=0 \text{ for all but finitely many primes }P\right\}.
	\end{align}
	If $a_{P} \neq 0$, we write $P \mid D$. The set of effective divisors is defined as 
	\begin{align}
		\mathfrak{D}^+:=\left\{\sum_{P\in\mathcal{P}}a_{P}P \in \mathfrak{D} : a_P \geq 0 \text{ for all primes }P\right\}.
	\end{align}
	For $A,B \in \mathfrak{D}$, we say $A\geq B$ if $A-B \in \mathfrak{D}$ is effective. This is analogous to $B\mid A$ in $F_q[T]$.
	
	For an effective divisor $D$, the Möbius function $\mu(D)$ is defined as:
	\begin{equation}
		\mu(D):=\begin{cases}
			1 & \text{if } D=0,\\
			(-1)^{k} & \text{if } D=P_{1}+\cdots+P_{k} \text{ for distinct primes } P_i \text{ and } k > 0,\\
			0 & \text{otherwise }.
		\end{cases}
	\end{equation}
	Following the same notation as in Section \ref{sec:intro}, we let $\mathcal{N}(D)$ denote the set of degrees of the prime factors of $D$ and let $\Omega(D)$ denote the cardinality of $\mathcal{N}(D)$.  Let $\Delta_k(D)$ and $\delta_k(D)$ denote the $k$-th largest and the $k$-th smallest elements of $\mathcal{N}(D)$ respectively, with $\Delta_k(D)=\delta_k(D)=0$ for $k>\Omega(D)$. Define
	\begin{align}
		P_{\min}(D) := \{\text{prime }P\mid D: \deg P = \delta_1(D)\}.
	\end{align}
	Generalizing \eqref{important_definitions}, for a subset $\mathcal{S}$ of the set of primes in $K$, we define
	\begin{align} 
		&\mathfrak{D}(K,\mathcal{S}) := \{D \in \mathfrak{D}^+ : P_{\min}(D) \text{ has only one element }P \text{ and } P\in\mathcal{S}\}, \text{ and}\\
		&Q_\mathcal{S}^{(k)}(D) := \#\{P\in \mathcal{S}:P\mid D,\deg P=\Delta_k(D)\}.
	\end{align}
	We are now ready to state higher order dualities over a global function field $K$. 
	
	\begin{theorem}\label{thm:global_duality}
		Let $K$ be a global function field and let $\mathcal{S}$ be a subset of primes in $K$. Let $f:\mathbb{N}\rightarrow\mathbb{C}$ be an arithmetic function with $f(0)=0$. For any effective divisor $A\in\mathfrak{D}^{+}$ and any positive integer $k$, we have
		\begin{equation}\label{higher duality in global ff}
			\sum_{\substack{B\in\mathfrak{D}(K,\mathcal{S}) \\ B \le A}}\mu(B)\binom{\Omega(B)-1}{k-1}f(\delta_{1}(B))=(-1)^{k}Q_{\mathcal{S}}^{(k)}(A)f(\Delta_{k}(A)).
		\end{equation}
	\end{theorem}
	
	\section{Surrounding literature}\label{sec:motivation}
	
	Alladi's original work \cite{alladi} initiated the use of the duality identity \eqref{alladi-duality} for $k=1$ to detect densities through conditions on the smallest prime factor. This framework was subsequently extended in several arithmetic settings. Dawsey \cite{dawsey} obtained a Chebotarev-density version of \eqref{alladi-main-eq} for finite Galois extensions $K/\mathbb{Q}$. Sweeting and Woo \cite{sweeting-woo} generalized this to arbitrary number fields $L/K$, and Kural, McDonald and Sah \cite{KMS} replaced the Chebotarev condition by arbitrary subsets of primes possessing natural densities. In the function field setting, Duan, Wang and Yi \cite{duan-wang-yi} established analogues of the results of Kural, McDonald and Sah over global function fields. In particular, their duality is the first-order identity which serves as the base case for our higher-order result.
	
	Generalizing the second identity in \eqref{alladi-johnson-eqs}, Tenenbaum \cite{tenenbaum2025} showed that the vanishing holds even if the arithmetic progression condition is replaced by any arbitrary subset of primes possessing a natural density. Sengupta \cite{sengupta} used the duality identity \eqref{alladi-duality} for $k=2$ to establish the result for Galois extensions $K/\mathbb{Q}$, proving the vanishing of weighted M\"{o}bius sums under Chebotarev conditions. In his recent work \cite{sengupta-new}, he extended this framework to arbitrary number fields and derived general higher-order duality identities for prime ideals. A fundamental distinction between his setting and ours is the choice of the counting function. Sengupta's identities involve $\omega_K(I)$, the number of distinct prime ideals in the factorization of an ideal $I$, whereas our identities over $\mathbb{F}_q[T]$ involve $\Omega(A)$, the number of distinct prime degrees occurring in the factorization of $A$.
	
	The study of polynomials free of large prime factors has substantial literature. Panario, Gourdon and Flajolet \cite{panario-gourdon-flajolet} developed an analytic treatment of smooth polynomials over finite fields and, in particular, studied the joint distribution of the two largest degrees occurring among the irreducible factors. Gorodetsky \cite{Gorodetsky2024Smooth} later obtained sharper uniform estimates for the usual smooth-polynomial counting function, while in the present paper we use a convenient uniform upper bound due to Thorne \cite{thorne}. Corresponding estimates have also been developed in the more general setting of arithmetical semigroups by Manstavi\v{c}ius \cite{manstivicius}. The distribution of distinct degree values in polynomial factorizations has also been studied earlier. Knopfmacher and Warlimont \cite{knopfmacher-warlimont} considered polynomials whose irreducible factors have pairwise distinct degrees, while Knopfmacher \cite{knopfmacher-degrees} determined the asymptotic mean number of distinct irreducible-factor degrees.

	Beyond density formulas of this type, Alladi's framework has been developed in several other directions. Wang extended these ideas using Ramanujan sums \cite{wang-ramanujan} and Dirichlet convolutions \cite{wang-dirichlet}. Ono, Schneider, and Wagner obtained partition-theoretic formulas connecting such densities with $q$-series in \cite{OSW-1} and \cite{OSW-2}. Duan, Ma, and Yi \cite{duan-ma-yi} generalized Alladi's formula to arithmetical semigroups, while Wang \cite{wang-logarithm} recently established a logarithmic analogue of the duality identity. Our work adds higher-order global function field dualities and their weighted M\"{o}bius applications to this circle of results.
	
	More recently, Alladi and Sengupta \cite{alladi-sengupta} have developed the classical theory to all higher orders. Using the $k$-th order duality and the distribution of the $k$-th largest prime factor in reduced residue classes, they further showed that, for $k\geq 3$,
	\begin{align}
		\sum_{\substack{n\geq 2\\ p_1(n)\equiv a\bmod m}}
		\frac{\mu(n)\omega(n)^{k-1}}{n}=0,
	\end{align}
	for positive integers $1\leq a\leq m$ with $\gcd(a,m)=1$. This suggests that our Theorem \ref{thm:mainthm AJ} can be generalized to allow for higher powers of $\Omega(A)$.
	We relegate this study to forthcoming work.
	
	\section{Prerequisites}\label{prelems}

	We recall Mertens' estimate over $\mathbb{F}_q[T]$.

	\begin{lemma}\cite[Lemma 2.1]{thorne}
		Let $\gamma$ denote the Euler's constant. Then, we have that
		\begin{equation}\label{merten}
			\prod_{\substack{\text{prime }P\\\deg P\leq n}}\big(1-q^{-\deg P}\big)^{-1}=ne^{\gamma}(1+o_n(1)).
		\end{equation}
	\end{lemma}
	The following result is a function field analogue of the prime number theorem.
	\begin{lemma} \cite[Theorem 5.12]{rosen}\label{lem:PNT in ff}
		Let $\pi_\mathcal{P}(n)$ denote the number of primes of degree $n$ in $\mathbb{F}_q[T]$. We have
		\begin{equation}\label{Lem: prime_num_thm_eq}
			\pi_\mathcal{P}(n)=\frac{q^n}{n}+O\left(\frac{q^{n/2}}{n}\right).
		\end{equation}
	\end{lemma}
	
	We state a well-known result for the M\"{o}bius function over $\mathbb{F}_q[T]$.
	
	\begin{lemma}\label{lem:sum mu in ff}
		For any monic $A\in\mathbb{F}_q[T]$, we have
		\begin{equation}
			\sideset{}{'}\sum_{B \mid A} \mu(B) = \begin{cases}
				1 & \text{if }A=1,\\
				0 & \text{otherwise}.
			\end{cases}
		\end{equation}
	\end{lemma}

	We now state an interesting result for weighted sums of  M\"{o}bius function over $\mathbb{F}_q[T]$.
	For a function 
	$f:\mathbb{R}\to\mathbb{R}$, 
	the $n$-th forward difference of $f$, $D_nf(x)$ is defined by the recurrence relation 
	$D_n f(x)=D_{n-1}f(x+1)-D_{n-1}f(x)$
	for $n\geq1$ with $D_0 f(x)=f(x)$. Then, by the principle of induction, it follows that
	\begin{equation}\label{eq:forward difference}
		D_nf(x)=\sum_{k=0}^n (-1)^{n-k} {n \choose k}f(x+k).
	\end{equation}
	Note that if $f$ is a real polynomial of degree $d$, then $D_nf(x)=0$ for all $n\geq d+1$.

	\begin{lemma}\label{lem: weighted sum mu in ff}
		For any monic $A\in\mathbb{F}_q[T]$ and any function $f:\mathbb{R}\to\mathbb{C}$, we have that
		\begin{equation}
			\sideset{}{'}\sum_{B \mid A} \mu(B)f(\Omega(B)) = (-1)^{\Omega(A)}D_{\Omega(A)}f(0).
		\end{equation}
	\end{lemma}
	\begin{proof}
		The statement holds trivially for $A=1$. We let $A$ be a non-constant monic polynomial in $\mathbb{F}_q[T]$. Then for any $B\mid A$, we have $\Omega(B)\leq\Omega(A)$. Hence, we can write
		\begin{align}\label{eq:weighted sum first step}
			\sideset{}{'}\sum_{B \mid A} 
			\mu(B)f(\Omega(B)) 
			=
			\sum_{k=0}^{\Omega(A)} f(k)
			\sideset{}{'}\sum_{\substack{B \mid A \\ \Omega(B)=k}} \mu(B).
		\end{align}
		We write
		$A=A_1A_2\cdots A_{\Omega(A)}$,
		where each $A_i$ is a product of primes of same degree. Then, any $B\mid A$ with $\Omega(B)=k$ can be written uniquely as $B_{i_1}B_{i_2}\cdots B_{i_k}$, where each $i_j$ is chosen distinctly from
		$\{1,2,\cdots \Omega(A)\}$
		and
		$1\neq B_{i_j}\mid A_{i_j}$ for $1\leq j\leq k$. Hence, by multiplicativity of $\mu$, we have
		\begin{align}
			\sideset{}{'}\sum_{\substack{B \mid A \\ \Omega(B)=k}} \mu(B)
			&=
			\sum_{\substack{\text{choose }i_1, i_2, \cdots, i_k\\ \text{from }\{1,2,\cdots,\Omega(A)\}}} \hspace{2mm} \sideset{}{'}\sum_{\substack{1\neq B_{i_j}\mid A_{i_j}\\ \text{for each }1\leq j\leq k}} \mu(B_{i_1})\mu(B_{i_2})\cdots\mu(B_{i_k})
			\\
			&=
			{\Omega(A) \choose k}\prod_{1\leq j\leq k}\hspace{2mm}\sideset{}{'}\sum_{1\neq B_{i_j}\mid A_{i_j}} \mu(B_{i_j})
			\\
			&= 
			{\Omega(A) \choose k} (-1)^k,
		\end{align}
		upon using Lemma \ref{lem:sum mu in ff}.
		Substituting this in \eqref{eq:weighted sum first step} and using the expression of $D_nf(x)$ given in \eqref{eq:forward difference}, we have
		\begin{align}
			\sideset{}{'}\sum_{B \mid A} \mu(B)f(\Omega(B))
			&= \sum_{k=0}^{\Omega(A)} f(k) {\Omega(A) \choose k} (-1)^k
			\\
			&= (-1)^{\Omega(A)}D_{\Omega(A)}f(0).
		\end{align}
	\end{proof}
	\begin{remark}
		The statement of Lemma \ref{lem: weighted sum mu in ff} holds even in the classical setting, more precisely, one has 
		\begin{align}
			\sum_{d\mid n}\mu(d)f(\omega(d))=(-1)^{\omega(n)}D_{\omega(n)}f(0),
		\end{align}
		where $n$ is a positive integer. This is of independent interest and follows by the same proof, though we do not use this for our results. 
	\end{remark}
	
	We note that taking $f$ to be identically one in Lemma \ref{lem: weighted sum mu in ff} yields Lemma \ref{lem:sum mu in ff}. Moreover, taking $f(x)=x$ yields the following corollary.
	\begin{corollary}\label{lem:sum muomega in ff}
		For any monic polynomial $A\in\mathbb{F}_q[T]$, we have that
		\begin{equation}
			\sideset{}{'}\sum_{B\mid A}\mu(B)\Omega(B) = \begin{cases}
				-1 & \text{if } A \text{ is a product of primes of same degree}, \\
				0 & \text{otherwise}.
			\end{cases}.
		\end{equation}
	\end{corollary}

	\begin{corollary}\label{lem:sum muomega gen in ff}
		For any monic $A$ in $\mathbb{F}_q[T]$ and for any integer $ \ell \geq 0$, we have
		\begin{equation}
			\sideset{}{'}\sum_{B \mid A}\mu(B)
			{\Omega(B)+1\choose \ell + 1} =
			\begin{cases}
				(-1)^{\Omega(A)} &\text{if } \Omega(A) =\ell \text{ or }\Omega(A)=\ell+1,\\
				0 &\text{otherwise}.
			\end{cases}
		\end{equation} 
	\end{corollary}
	\begin{proof}
		If $\Omega(A)\leq \ell-1$, then for any monic $B \mid A$, we have $0\leq \Omega(B)\leq \ell-1$ and hence 
		\begin{align}
			{\Omega(B)+1\choose \ell + 1}=0.
		\end{align}
		As a result, we have that whenever $\Omega(A)\leq \ell-1$, 
		\begin{equation}
			\sideset{}{'}\sum_{B \mid A}\mu(B)
			{\Omega(B)+1\choose \ell + 1}=0.
		\end{equation}
		
		Now, consider the function $f_\ell:\mathbb{R}\to\mathbb{R}$, defined as $f_\ell(x)={x+1\choose \ell + 1}$.
		We note that $f_\ell(x)$ is just a polynomial of degree $\ell +1$. We consider three cases.
		
		\textbf{Case 1 : }If $\Omega(A)\geq \ell+2$, then $D_{\Omega(A)}f_\ell$ is identically zero.
		
		\textbf{Case 2 : }If $\Omega(A)=\ell+1$, 
		then by \eqref{eq:forward difference},
		\begin{equation}\label{eq:l+1}
			D_{\Omega(A)}f_\ell(0) = \sum_{k=0}^{\Omega(A)} (-1)^{\Omega(A)-k}{\Omega(A) \choose k} {k+1 \choose \Omega(A)}.
		\end{equation}
		Since every summand on the right hand side of \eqref{eq:l+1} vanishes for every $k\leq \Omega(A)-2$, we get
		\begin{equation}
			D_{\Omega(A)}f_\ell(0) = -\Omega(A)+\Omega(A)+1 = 1.
		\end{equation}
		
		\textbf{Case 3 : } If $\Omega(A)=\ell$, then by \eqref{eq:forward difference},
		\begin{equation}\label{eq:l+2}
			D_{\Omega(A)}f_\ell(0) = \sum_{k=0}^{\Omega(A)} (-1)^{\Omega(A)-k}{\Omega(A) \choose k} {k+1 \choose \Omega(A)+1}.
		\end{equation}
		Since every summand on the right hand side of \eqref{eq:l+2} vanishes for every $k\leq \Omega(A)-1$, we get
		\begin{equation}
			D_{\Omega(A)}f_\ell(0) = 1.
		\end{equation}
		
		hence by Lemma \ref{lem: weighted sum mu in ff}, the lemma is proved.
		
	\end{proof}
	
	\normalcolor
	
	Recall the definition of $Q_\mathcal{S}^{(1)}(A)$ from \eqref{important_definitions}. We state an important asymptotic result by Duan, Wang and Yi  about the growth of  $Q_\mathcal{S}^{(1)}(A)$ for monics $A$ of degree $n$. This will be instrumental in obtaining an equivalent statement for Theorem \ref{thm:mainthm AJ}.
	\begin{lemma} \cite[Theorem 4.4]{duan-wang-yi}\label{lem:QS1}
		We have that
		\begin{equation}\label{eq:QS1}
			\frac{1}{q^n}\sideset{}{'}\sum_{\deg A = n}Q_\mathcal{S}^{(1)}(A)\to\delta(\mathcal{S}),
		\end{equation}
		as $n \rightarrow \infty$.
	\end{lemma}
	
	A $y$-smooth number is a positive integer with no prime factor greater than $y$. Alladi's proof of \eqref{alladi-main-eq} involves using a suitable bound by de Bruijn \cite{bruijn}, later refined by Hildebrand and Tenenbaum \cite{HJ}, for the quantity 
	\begin{align}
		\psi_1(x,y)
		&:=\#\{n\leq x : n \text{ is }y\text{-smooth}\}\\
		&=\#\{n\leq x : P_1(n)\leq y\}.\label{eq:psi1}
	\end{align}
	Similarly, a second order quantity 
	$\psi_2(x,y) = \#\{n\leq x : P_2(n)\leq y\}$
	appears in the proof of the second identity in \eqref{alladi-johnson-eqs}. In the function fields context, an $m$-smooth polynomial in $\mathbb{F}_q[T]$ is a monic with no prime factor having degree greater than $m$. Analogous to \eqref{eq:psi1}, we have the definition
	\begin{align}
		\Psi_1(n,m):=\#\{A \in \mathbb{F}_q[T] : A \text{ is monic}, \deg A = n, \Delta_1(A)\leq m\}.
	\end{align}
	In the proof of Theorem \ref{thm:mainthm AJ}, we will need to estimate the quantity 
	\begin{align}
		\Psi_2(n,m):=\#\{A \in \mathbb{F}_q[T] : A \text{ is monic}, \deg A = n, \Delta_2(A)\leq m\}.
	\end{align}
	Our estimation of $\Psi_2(n,m)$ relies on suitable bounds for $\Psi_1(n,m)$, which need to be uniform for $1\leq m\leq n$. While there are many refined bounds (Manstavi\v{c}ius \cite{manstivicius}, Gorodetsky \cite{Gorodetsky2024Smooth}), in the interest of avoiding unwieldy calculations, we use the following elegant bound by 
	Thorne \cite{thorne}.
	
	\begin{lemma}{\cite[Lemma 2.3]{thorne}}
		Let $\Psi(r,s)$ denote the number of monic polynomials in
		$\mathbb{F}_q[T]$
		of degree at most $r$ with all prime factors having degree at most $s$. Then, we have
		\begin{equation}
			\Psi(r,s) \ll q^r s e^{-r/s}.
		\end{equation} 
	\end{lemma}
	Our function $\Psi_1(r,s)$ counts polynomials of degree $r$ with all prime factors having degree at most $s$. This is bounded above by $\Psi(r,s)$, so that 
	\begin{equation}\label{eq : Psi1 in ff}
		\Psi_1(r,s) \ll q^r s e^{-r/s}.
	\end{equation}
	We now state a lemma instrumental in proving Theorem \ref{thm:Psi2}.

	\begin{lemma}\label{lem:useful}
		For every integer $d\geq 1$, We have that
		\begin{align}
			\sideset{}{'}\sum_{\substack{\Delta_1(A)=d}} \frac{\deg A}{q^{\deg A}}\ll d.
		\end{align}
	\end{lemma}
	\begin{proof}
		Every $A$ with $\Delta_1(A)=d$ can be written as $A=PB$, where $P$ is a prime with $\deg P=d$ and $\Delta_1(B)\leq d$. So, we have that
		\begin{align}
			\sideset{}{'}\sum_{\substack{\Delta_1(A)=d}} \frac{\deg A}{q^{\deg A}} 
			&\leq \sideset{}{'}\sum_{\substack{P \text{ prime}\\ \deg P = d}} \sideset{}{'}\sum_{\substack{\Delta_1(B)\leq d}} \frac{d + \deg B}{q^{d+\deg B}}\\
			&=\sideset{}{'}\sum_{\substack{P \text{ prime}\\ \deg P = d}} \frac{d}{q^{d}} \sideset{}{'}\sum_{\substack{\Delta_1(B)\leq d}} \frac{1}{q^{\deg B}} 
			+ \sideset{}{'}\sum_{\substack{P \text{ prime}\\ \deg P = d}} \frac{1}{q^{d}} \sideset{}{'}\sum_{\substack{\Delta_1(B)\leq d}} \frac{\deg B}{q^{\deg B}}.
		\end{align}
		Hence, by Lemma \ref{lem:PNT in ff},
		\begin{align}\label{plgin}
			\sideset{}{'}\sum_{\substack{\Delta_1(A)=d}} \frac{\deg A}{q^{\deg A}} 
			&\ll  \sideset{}{'}\sum_{\substack{\Delta_1(B)\leq d}} \frac{1}{q^{\deg B}} 
			+ \frac{1}{d} \sideset{}{'}\sum_{\substack{\Delta_1(B)\leq d}} \frac{\deg B}{q^{\deg B}}.
		\end{align}
		We now consider the function $F_d:\mathbb R \to \mathbb R$, defined as
		\begin{align}
			F_d(x):=\sideset{}{'}\sum_{\substack{\Delta_1(B)\leq d}} \frac{x^{\deg B}}{q^{\deg B}} = \prod_{\substack{P \text{ prime}\\ \deg P \leq d}} \left(1-\frac{x^{\deg P}}{q^{\deg P}}\right)^{-1}.
		\end{align}
		By Mertens' estimate \eqref{merten}, we have that
		\begin{align}\label{firsum}
			\sideset{}{'}\sum_{\substack{\Delta_1(B)\leq d}} \frac{1}{q^{\deg B}} = F_d(1) \ll d.
		\end{align}
		We also observe that
		\begin{align}
			\frac{F_d'(x)}{F_d(x)}
			&=  \frac{d (\log F_d(x))}{dx}
			= - \sum_{\substack{P \text{ prime}\\ \deg P \leq d}} \frac{d}{dx}  \log \left(1-\frac{x^{\deg P}}{q^{\deg P}}\right)\\
			&= - \sum_{\substack{P \text{ prime}\\ \deg P \leq d}} \frac{-\deg P \cdot x^{\deg P -1} q^{-\deg P} }{1 - x^{\deg P}q^{-\deg P}},
		\end{align}
		and hence
		\begin{align}\label{secsum}
			\sideset{}{'}\sum_{\substack{\Delta_1(B)\leq d}} \frac{\deg B}{q^{\deg B}} = F_d'(1) &= F_d(1) \cdot \sum_{\substack{P \text{ prime}\\ \deg P \leq d}} \frac{\deg P \cdot q^{-\deg P} }{1 - q^{-\deg P}}\\
			&\ll d \cdot \sum_{1\leq t \leq d} \frac{q^t}{t} \frac{tq^{-t}}{1-q^{-t}} \ll d^2.
		\end{align}
		Plugging \eqref{firsum} and \eqref{secsum} in \eqref{plgin}, we obtain that
		\begin{align}
			\sideset{}{'}\sum_{\substack{\Delta_1(A)=d}} \frac{\deg A}{q^{\deg A}}\ll d.
		\end{align}
	\end{proof}

	We end this section by stating a lemma by Duan, Wang, and Yi \cite{duan-wang-yi}, which is useful in  proving Theorem \ref{thm:mainthm AJ}. 
	
	\begin{lemma}{\cite[Lemma 4.5]{duan-wang-yi}} \label{lem:vs}
		(a) Suppose $f,g: \mathbb{N}\to(1,\infty)$ are two arithmetic functions such that
		$$\lim_{n\to\infty}g(n)=\infty \text{ and } \lim_{n\to\infty}\frac{f(n)}{g(n)}=\delta.$$
		Suppose $g(n)$ is monotonically increasing. 
		Define
		$$e(n;f,g):=\sup_{k\le n}|f(k)-\delta g(k)|.$$
		Then 
		\begin{equation}
			\lim_{n\to\infty}\frac{e(n;f,g)}{g(n)}=0.
		\end{equation}
		(b) Let
		$f: \mathbb{N}\to(0,\infty)$ 
		be a decreasing function with
		$\lim_{n\to\infty}nf(n)=0$. 
		Suppose 
		$h: \mathbb{N}\to[1,\infty)$ is an arbitrary function such that $\lim_{n\to\infty}h(n)=\infty$ and $\lim_{n\to\infty}n/h(n)=\infty$ .
		Then  there exists a  sequence $\{k_n\}_{n=1}^\infty$ of positive integers satisfying $\lim_{n\to\infty}k_n=\infty$ such that 
		\begin{enumerate}
			\item $\lim_{n\to\infty}n/k_n=\infty$.
			\item $n/k_n\ll h(n)$. 
			\item $\lim_{n\to\infty}n\cdot f(k_n)=0$.
		\end{enumerate}
	\end{lemma}
	
	\begin{remark}
		Lemma \ref{lem:vs} also holds when $f:\mathbb{N}\to[0,\infty)$ and
		$g:\mathbb{N}\to(0,\infty)$ and the proof follows through as outlined in \cite[Lemma 4.5]{duan-wang-yi}.
	\end{remark}
	
	We are now well-equipped to prove our main results.

	\section{Proof of Theorem \ref{thm:duality in ff}}\label{sec:thm1pf}
	
	The statement holds trivially for $A=1$. Let us assume that $A$ is some non-constant monic polynomial.	Note that if $k>\Omega(A)$, then the duality identity \eqref{gen duality in ff} holds trivially as $f(0)=0$. So, without any loss of generality, we let $k\leq \Omega(A)$.
	
	We apply induction on $k$ to prove the result.
	For $k=1$, the duality identity \eqref{gen duality in ff} reduces to \eqref{duan-duality}, which has been proved by Duan, Wang, and Yi \cite{duan-wang-yi}.  Assume that \eqref{gen duality in ff} holds for some integer $1 \leq r \leq \Omega(A)-1$.  So, we have
	\begin{equation}\label{eq:induction_hypothesis}
		\sideset{}{'}\sum_{B\in\mathcal{D}(\mathcal{S}), \hspace{1mm} B\mid A} \mu(B){\Omega(B)-1\choose r-1}f(\delta_1(B))
		=
		(-1)^rQ_\mathcal{S}^{(r)}(A)f(\Delta_r(A)).
	\end{equation}
	We consider the sum
	\begin{align}
		&\sideset{}{'}\sum_{B\in\mathcal{D}(\mathcal{S}), \hspace{1mm} B\mid A} \mu(B){\Omega(B)-1\choose r}f(\delta_1(B))
		+
		(-1)^rQ_\mathcal{S}^{(r)}(A)f(\Delta_r(A))
		\\
		&=\sideset{}{'}\sum_{B\in\mathcal{D}(\mathcal{S}), \hspace{1mm} B\mid A} \mu(B)\left\{{\Omega(B)-1\choose r}+{\Omega(B)-1\choose r-1}\right\}f(\delta_1(B))\\
		&=\sideset{}{'}\sum_{B\in\mathcal{D}(\mathcal{S}), \hspace{1mm} B\mid A} \mu(B){\Omega(B)\choose r}f(\delta_1(B)),
	\end{align}
	where the last equality follows by Pascal's identity. Hence, we have that	
	\begin{align}
		&\sideset{}{'}\sum_{B\in\mathcal{D}(\mathcal{S}), \hspace{1mm} B\mid A} \mu(B){\Omega(B)-1\choose r}f(\delta_1(B))
		+
		(-1)^rQ_\mathcal{S}^{(r)}(A)f(\Delta_r(A))\\
		&=\frac{1}{r!}\sideset{}{'}\sum_{B\in\mathcal{D}(\mathcal{S}), \hspace{1mm} B\mid A} \mu(B)\Omega(B)\Big(\Omega(B)-1\Big)\Big(\Omega(B)-2\Big)\cdots\Big(\Omega(B)-r+1\Big)f(\delta_1(B)).\label{former}
	\end{align}
	We proceed to simplify the sum in \eqref{former}.
	Let $\Omega(A)=m$ and $\mathcal{N}(A)=\{d_1,d_2,\cdots, d_m\}$, where $d_1<d_2<\cdots< d_m$. Write $A=A_1A_2\cdots A_m$ where each $A_i$ is the product of all the  prime divisors of $A$ of degree $d_i$ (with multiplicity). For a monic divisor $B$ of $A$, write $B=B_1B_2\cdots B_m$ where each $B_i\mid A_i$. Note that $B_i$ can be $1$. 
	
	Let $1\leq t(B)\leq m$ be the smallest index such that $B_{t(B)}\neq1$. Then, any $B\in\mathcal{D}(\mathcal{S})$ can be uniquely written as $B=P_{t(B)}^{\alpha_{t}}B'$ for some positive integer $\alpha_t$, where $P_{t(B)} \in \mathcal{S}$ is some prime divisor of $A_{t(B)}$ and $B'$ is some monic divisor of $A_{t(B)+1}\cdots A_m$. Note that $\mu(P_{t(B)}^{\alpha_{t}}B')=-\mu(B')$ if $\alpha_t=1$, and $0$ otherwise. Then, we have

	\begin{align}
		&\sideset{}{'}\sum_{B\in\mathcal{D}(\mathcal{S}), \hspace{1mm} B\mid A} \mu(B)\Omega(B)\Big(\Omega(B)-1\Big)\Big(\Omega(B)-2\Big)\cdots\Big(\Omega(B)-r+1\Big)f(\delta_1(B))\\
		&= \sum_{t=1}^m f(d_t) \sideset{}{'}\sum_{P \in \mathcal{S}: P \mid A_t}\mu(P)\sideset{}{'}\sum_{B' \mid A_{t+1}\cdots A_m} \mu(B')\Big(\Omega(B')+1\Big)\Omega(B')\Big(\Omega(B')-1\Big)\cdots\Big(\Omega(B')-r+2\Big)\\
		&= -\sum_{t=1}^m f(d_t) Q_\mathcal{S}^{(m-t+1)}(A) \sideset{}{'}\sum_{B' \mid A_{t+1}\cdots A_m} \mu(B')\Big(\Omega(B')+1\Big)\Omega(B')\Big(\Omega(B')-1\Big)\cdots\Big(\Omega(B')-r+2\Big)\\
		&=  -r!\sum_{t=1}^m f(d_t) Q_\mathcal{S}^{(m-t+1)}(A) \sideset{}{'}\sum_{B' \mid A_{t+1}\cdots A_m} \mu(B'){\Omega(B')+1 \choose r}.
	\end{align}
	
	By Corollary \ref{lem:sum muomega gen in ff}, since $r\geq 1$,
	\begin{align}
		\sum_{B' \mid A_{t+1}\cdots A_m} \mu(B'){\Omega(B')+1 \choose r}&=
		\begin{cases}
			(-1)^{m-t} & \text{if } m-t=r-1 \text{ or } m-t=r,\\
			0 & \text{otherwise}.
		\end{cases}\\
		&=
		\begin{cases}
			(-1)^{r-1} & \text{if } t=m-r+1,\\
			(-1)^{r} & \text{if } t=m-r,\\
			0 & \text{otherwise}.
		\end{cases}
	\end{align} 
	\normalcolor
	Hence, we have that 
	\begin{equation}
		\begin{split}
			&\sideset{}{'}\sum_{B\in\mathcal{D}(\mathcal{S}), \hspace{1mm} B\mid A} \mu(B)\Omega(B)\Big(\Omega(B)-1\Big)\Big(\Omega(B)-2\Big)\cdots\Big(\Omega(B)-r+1\Big)f(\delta_1(B))\\
			&= - f(d_{m-r+1})Q_\mathcal{S}^{(r)}(A)(-1)^{r-1}r! - f(d_{m-r})Q_\mathcal{S}^{(r+1)}(A)(-1)^{r}r!\\
			&= - f(\Delta_r(A))Q_\mathcal{S}^{(r)}(A)(-1)^{r-1}r! -f(\Delta_{r+1}(A))Q_\mathcal{S}^{(r+1)}(A)(-1)^{r}r!.\label{latter}
		\end{split}
	\end{equation}
	Substituting \eqref{latter} in \eqref{former}, we get
	\begin{equation}
		\begin{split}
			&\sideset{}{'}\sum_{B\in\mathcal{D}(\mathcal{S}), \hspace{1mm} B\mid A} \mu(B){\Omega(B)-1\choose r}f(\delta_1(B))\\
			&= - f(\Delta_r(A))Q_\mathcal{S}^{(r)}(A)(-1)^{r-1} -f(\Delta_{r+1}(A))Q_\mathcal{S}^{(r+1)}(A)(-1)^{r}-(-1)^rQ_\mathcal{S}^{(r)}(A)f(\Delta_r(A))\\
			&= (-1)^{r+1}f(\Delta_{r+1}(A))Q_\mathcal{S}^{(r+1)}(A).
		\end{split}
	\end{equation}
	Thus, by  induction, the theorem is proved. 
	
	\section{Proof of Theorem \ref{thm:Psi2}}
	Recall that 
	\begin{align}
		\Psi_2(n,m) = \#S_2(n,m) = \# \{\text{monic }A \in \mathbb{F}_q[T]: \deg A = n, \Delta_2(A)\leq m\}.
	\end{align}
	We begin by obtaining a bound for
	\begin{align}
		\mathcal{W}(n)&=\{\text{monic }A \in \mathbb{F}_q[T]: \deg A = n, \Delta_2(A)=0\}\\
		&=\{\text{monic }A \in \mathbb{F}_q[T]: \deg A = n, \Omega(A)=1\}.
	\end{align}
	\begin{lemma}\label{lem:N}
		Let $0<\epsilon<1$ be arbitrary. We have that for $x \in \mathbb{N}$,
		\begin{align}
			\#\mathcal{W}(x)\ll \tau(x)\frac{q^x}{x} \ll_\epsilon \frac{q^x}{x^{1-\epsilon}},
		\end{align}
		where $\tau(x)$ denotes the number of divisors of x. In particular, $\#\mathcal{W}(x)=o(q^x)$. 
	\end{lemma}
	\begin{proof}
		We note that
		\begin{align}\label{sum_W}
			\#\mathcal{W}(x)= \sum_{d\mid x}\#\mathcal{W}_d(x),
		\end{align}
		where $\mathcal{W}_d(x)$ is the set of monics $B$ in $\mathbb{F}_q[T]$ with $\deg B = x$, such that $\deg P = d$ for all prime divisors $P$ of $B$. It thus suffices to show that for all $d\mid x$, $$\#\mathcal{W}_d(x) \ll q^x/x.$$
		Let $k=\pi_{\mathcal{P}}(d)$. We denote the set of all primes of degree $d$ as
		\begin{align}
			\mathbb{P}_d = \{P_1, P_2, ... P_k\}.
		\end{align}
		The set $\mathcal{W}_d(x)$ is of the form
		\begin{align}
			\mathcal{W}_d(x) = \{P_1^{e_1} P_2^{e_2} \cdots P_k^{e_k} : e_i \geq 0, e_1 + e_2 + \cdots + e_k = x/d \}.
		\end{align}
		This is in bijective correspondence with the set 
		\begin{align}\label{stars_bars}
			\{(e_1, e_2, \cdots, e_k) : e_i \geq 0, e_1 + e_2 + \cdots + e_k = x/d \}.
		\end{align}
		\textbf{Case 1 :} Assume $d < \frac{1}{2}\log_q x$. By \eqref{stars_bars}, the set $\mathcal{W}_d(x)$  sits injectively inside the set $$\{(e_1, e_2, \cdots, e_k) : 0\leq  e_i \leq x/d \}.$$ 
		Since every zero of a prime of degree $d$ belongs to $\mathbb{F}_{q^d}$ and no two primes have a common zero, we have that $d\pi_{\mathcal{P}}(d) \leq q^d$. Hence, we have
		\begin{align}
			\#\mathcal{W}_d(x)
			\leq
			\left(1+\frac{x}{d}\right)^k
			\leq
			\left(\frac{2x}{d}\right)^{q^d/d}.
		\end{align}
		Thus, as $d<\frac{1}{2}\log_q x$, we obtain
		\begin{align}
			\log_q\#\mathcal{W}_d(x)
			\leq
			\frac{q^d}{d}
			\log_q\left(\frac{2x}{d}\right)
			\ll
			\sqrt{x}\log_q x
			\ll
			\frac{x}{2},
		\end{align}
		for all sufficiently large $x$. As a result,
		\begin{align}
			\#\mathcal{W}_d(x)\ll q^{x/2}\ll \frac{q^x}{x}.
		\end{align}
		\textbf{Case 2 :} 
		Assume $\frac{1}{2}\log_q x \leq d < \log_q x$. By \eqref{stars_bars}, the size of $\mathcal{W}_d(x)$ is exactly given by the formula 
		\begin{align}
			\# \mathcal{W}_d(x) = {k + x/d - 1\choose x/d}.
		\end{align}
		Since $k= \pi_\mathcal{P}(d) \leq \frac{q^d}{d}$ and $d < \log_q x$, we have that
		\begin{align}
			\# \mathcal{W}_d(x) \leq {2x/d\choose x/d} \leq 4^{x/d}.
		\end{align}
		Hence using $\frac{1}{2}\log_q x \leq d$, we get that
		\begin{align}
			\log_q \# \mathcal{W}_d(x) \ll \frac{x}{d} \leq \frac{2x}{\log_q x} \ll \frac{x}{2}.
		\end{align}
		As a result, $$\#\mathcal{W}_d(x) \ll q^{x/2}\ll \frac{q^x}{x}.$$
		\textbf{Case 3 :} Assume $d \geq \log_q x$. 
		We use the inequality ${u \choose v} \leq \frac{u^v}{v!}$ to write 
		\begin{align}
			\# \mathcal{W}_d(x) \leq \frac{(k + x/d-1)^{x/d}}{(x/d)!} \leq \frac{(k + x/d)^{x/d}}{(x/d)!}.
		\end{align}
		Since  $d \geq \log_q x$, we have $x \leq q^d$ and hence by Lemma \ref{lem:PNT in ff}, 
		\begin{align}
			\# \mathcal{W}_d(x) \leq \frac{(2q^d/d)^{x/d}}{(x/d)!} = q^x \frac{2^{x/d}}{d^{x/d}(x/d)!} \ll \frac{q^x}{x},
		\end{align}
		where the last bound follows by the Stirling inequality $n! \ge \left(\frac{n}{e}\right)^n$. Combining the three cases and using the fact that $\tau(x)\ll_\epsilon x^\epsilon$ for arbitrary $\epsilon>0$, the lemma follows.
	\end{proof}

	Thus by Lemma \ref{lem:N}, we can say that for any arbitrary $0<\epsilon<1$, we have the bound
		\begin{align}
			\#\{\text{monic }A \in \mathbb{F}_q[T]: \deg A = n, \Delta_2(A)=0\}=\#\mathcal{W}(n)\ll_\epsilon \frac{q^n}{n^\epsilon}.
		\end{align}
		Hence, it now only suffices to estimate the count
		\begin{align}
			\#S_2^+(n,m) = \# \{\text{monic }A \in \mathbb{F}_q[T]: \deg A = n, 1\leq \Delta_2(A)\leq m\}.
		\end{align}
		Let $M(n)$ be the set of all $n$-degreed polynomials having multiple (including multiplicity) prime factors  of largest degree and let $U(n,m)=S_2^+(n,m)\setminus M(n)$. We show that 
		\begin{align}
			\#M(n) \ll_\epsilon \frac{q^n}{n^\epsilon} \quad \text{and} \quad
			\#U(n,m) \ll q^n\frac{m}{n}.
		\end{align}
		To estimate $\#M(n)$, we first note that by Thorne's bound \eqref{eq : Psi1 in ff}, for any $0<\epsilon<1$,
		\begin{equation}\label{S1}
			\#S_1(n,n^\epsilon) = \Psi_1(n,n^\epsilon) \ll q^n n^\epsilon e^{-n^{1-\epsilon}}.
		\end{equation}
		For the remainder set $R(n)=M(n)\setminus S_1(n,n^\epsilon)$, we note that any $A\in R(n)$ can be written as $BQQ'$, where $Q,Q'$ are primes with $\deg Q = \deg Q' = \Delta_1(A)>n^\epsilon$. Hence, we have the inequality
		\begin{equation}
			\begin{split}
				\#R(n) \leq  \sum_{d>n^\epsilon}\sideset{}{'}\sum_{\substack{\text{primes }Q,Q'\\ \deg Q = \deg Q' = d}}\sideset{}{'}\sum_{\deg B=n-2d}1.
			\end{split}
		\end{equation}
		By  Lemma \ref{lem:PNT in ff} and the fact that there are $q^{n-2d}$ monics of degree $n-2d$, we obtain 
		\begin{equation}\label{Rn}
			\#R(n) \ll \sum_{d>n^\epsilon} \frac{q^{2d}}{d^2}q^{n-2d} = q^n \sideset{}{'}\sum_{d>n^\epsilon} \frac{1}{d^2} \ll \frac{q^n}{n^\epsilon}.
		\end{equation}
		Combining \eqref{S1} and $\eqref{Rn}$, we get
		\begin{equation}\label{eq: Mn}
			\#M(n) \ll \frac{q^n}{n^\epsilon}+q^n n^\epsilon e^{-n^{1-\epsilon}} \ll \frac{q^n}{n^\epsilon}.
		\end{equation}
		
		We now estimate the size of the set $U(n,m)=S_2^+(n,m)\setminus M(n)$. We write $U(n,m)$ as a disjoint union of $U_d(n,m)$ for $1\leq d\leq m$,
		where $$U_d(n,m)=\{A\in U(n,m) : \Delta_2(A)=d\}.$$ For each $A\in U_d(n,m)$, write $A=BQ$ where $Q$ is the unique prime divisor of $A$ with $\deg Q=\Delta_1(A)$. This yields the bijective correspondence
		\begin{equation}
			\Phi:U_d(n,m) \to \{(B,Q):B \text{ monic with }\deg B<n-d, \Delta_1(B)= d, Q \text{ prime with}\deg Q = n-\deg B\}.
		\end{equation}
		Hence, we have the identity
		\begin{equation}
			\# U_d(n,m) = \sideset{}{'}\sum_{\substack{\deg B < n-d\\ \Delta_1(B)=d}} \sideset{}{'}\sum_{\substack{\text{prime }Q\\ \deg Q =n-\deg B}}1. 
		\end{equation}
		Employing  Lemma \ref{lem:PNT in ff} and splitting the first sum at $n/2$ gives
		\begin{equation}\label{eq: Ud split}
			\# U_d(n,m) \ll \sideset{}{'}\sum_{\substack{\deg B \leq  n/2\\ \Delta_1(B)=d}} \frac{q^{n-\deg B}}{n-\deg B} + \sideset{}{'}\sum_{\substack{n/2 < \deg B < n-d\\ \Delta_1(B)=d}} \frac{q^{n-\deg B}}{n-\deg B}.
		\end{equation}
		For $\deg B \leq n/2$, we have $n-\deg B \geq n/2$ and hence the first sum in \eqref{eq: Ud split} becomes
		\begin{equation}
			\sideset{}{'}\sum_{\substack{\deg B \leq  n/2\\ \Delta_1(B)=d}} \frac{q^{n-\deg B}}{n-\deg B} \ll \frac{q^n}{n} \sideset{}{'}\sum_{\substack{\deg B \leq  n/2\\ \Delta_1(B)=d}}q^{-\deg B}.
		\end{equation}
		For the second sum in \eqref{eq: Ud split}, we have $n-\deg B > d$ and so 
		\begin{equation}
			\sideset{}{'}\sum_{\substack{n/2 <\deg B < n-d\\ \Delta_1(B)=d}} \frac{q^{n-\deg B}}{n-\deg B} < \frac{q^n}{d} \sideset{}{'}\sum_{\substack{n/2 <\deg B < n-d\\ \Delta_1(B)=d}}q^{-\deg B}.
		\end{equation}
		Since $U(n,m)$ is a disjoint union of $U_d(n,m)$, we hence have the estimate
		\begin{align}
			\# U(n,m) &\ll q^n \left(\frac{1}{n}\sum_{d\leq m} \sideset{}{'}\sum_{\substack{\deg B \leq  n/2\\ \Delta_1(B)=d}}q^{-\deg B}  + \sum_{d\leq m}\frac{1}{d}\sideset{}{'}\sum_{\substack{n/2 <\deg B < n-d\\ \Delta_1(B)=d}}q^{-\deg B}\right)\\
			&\leq q^n \left(\frac{1}{n}\sum_{d\leq m} \sideset{}{'}\sum_{\substack{\deg B \leq  n/2\\ \Delta_1(B)=d}}q^{-\deg B}  + \sum_{d\leq m}\frac{1}{d}\sideset{}{'}\sum_{\substack{n/2 <\deg B\\ \Delta_1(B)=d}}\frac{\deg B}{n/2}q^{-\deg B}\right).
		\end{align}
		
		By Lemma \ref{lem:useful}, we can simplify this to write
		\begin{equation}
			\begin{split}
				\# U(n,m) 
				&\ll q^n \left(\frac{1}{n} \sideset{}{'}\sum_{\substack{\deg B \leq  n/2\\ \Delta_1(B)\leq m}}q^{-\deg B}  + \sum_{d\leq m}\frac{1}{n}\right)\\
				&\leq q^n \left(\frac{1}{n} \sideset{}{'}\sum_{\substack{\deg B \leq  n/2\\ \Delta_1(B)\leq m}}q^{-\deg B}  + \frac{m}{n}\right).
			\end{split}
		\end{equation}
		Using Mertens' estimate \eqref{merten}, we obtain
		\begin{align}\label{eq:Un}
			\# U(n,m)\ll q^n\frac{m}{n}.
		\end{align}
		Combining equations \eqref{eq: Mn} and \eqref{eq:Un} proves the theorem.

	\section{Proof of Theorem \ref{thm:landau_two_FF}}
	
		We write 
		\begin{align}
			\sideset{}{'}\sum_{1\leq \deg A\leq x} \frac{\mu(A)\Omega(A)}{q^{\deg A}}
			&= \frac{1}{q^x}\sideset{}{'}\sum_{1\leq \deg A\leq x} \mu(A)\Omega(A)q^{x-\deg A}\\
			&= \frac{1}{q^x}\sideset{}{'}\sum_{1\leq \deg A\leq x} \mu(A)\Omega(A)\sideset{}{'}\sum_{\substack{\deg B = x\\ A\mid B}} 1.		
		\end{align}
		Interchanging the order of summation, we get 
		\begin{align}
			\sideset{}{'}\sum_{1\leq \deg A\leq x} \frac{\mu(A)\Omega(A)}{q^{\deg A}} 
			= \frac{1}{q^x}\sideset{}{'}\sum_{\deg B=x} \sideset{}{'}\sum_{A\mid B} \mu(A)\Omega(A).
		\end{align}
		By Corollary \ref{lem:sum muomega in ff}, the above equality simplifies to 
		\begin{align}
			\sideset{}{'}\sum_{1\leq \deg A\leq x} \frac{\mu(A)\Omega(A)}{q^{\deg A}}
			= - \frac{\#\mathcal{W}(x)}{q^x}.
		\end{align}
		Hence, by Lemma \ref{lem:N}, the theorem follows.

	\section{Proof of Theorem \ref{thm:mainthm AJ}}
	
	Our first step in this direction is obtaining
	an equivalence result which relies on the duality identity and Lemma \ref{lem:QS1}.

	\begin{lemma}\label{lem:equivalence}
		We have that 
		\begin{align}
			\lim_{x\to\infty}\sideset{}{'}\sum_{\substack{1\leq\deg A\leq x\\A\in\mathcal{D}(\mathcal{S})}}\frac{\mu(A)\Omega(A)}{q^{\deg A}}=0
		\end{align}
		holds if and only if 
		\begin{align}
			\frac{1}{q^n}\sideset{}{'}\sum_{\deg A = n}Q_\mathcal{S}^{(2)}(A)\to\delta(\mathcal{S}),
		\end{align}
		as $n \to \infty$. 
	\end{lemma}

	\begin{proof}
		We consider the general duality identity from Theorem \ref{thm:duality in ff} specifically when $f(n)=1$ for all positive integers $n$. Adding the identities for $k=1$ and $k=2$, we obtain
		\begin{equation}
			\sideset{}{'}\sum_{B\in\mathcal{D}(\mathcal{S}), \hspace{1mm} B\mid A} \mu(B)\Omega(B)=Q_\mathcal{S}^{(2)}(A)-Q_\mathcal{S}^{(1)}(A).
		\end{equation}
		Summing over all monics $A$ with degree $n$ and dividing by $q^n$, we have
		\begin{equation}\label{eq:equivalence_eq1}
			\frac{1}{q^n}\sideset{}{'}\sum_{\deg A=n}\hspace{2mm}\sideset{}{'}\sum_{B\in\mathcal{D}(\mathcal{S}), \hspace{1mm} B\mid A} \mu(B)\Omega(B)=\frac{1}{q^n}\sideset{}{'}\sum_{\deg A=n}\left(Q_\mathcal{S}^{(2)}(A)-Q_\mathcal{S}^{(1)}(A)\right).
		\end{equation}
		Changing the order of summation on the left hand side gives us 
		\begin{equation}\label{eq:equivalence_eq2}
			\frac{1}{q^n}\sideset{}{'}\sum_{\substack{1\leq \deg B\leq n\\ B\in\mathcal{D}(\mathcal{S})}}\mu(B)\Omega(B)\sideset{}{'}\sum_{\substack{\deg A =n\\ B\mid A}} 1 = \frac{1}{q^n}\sideset{}{'}\sum_{\substack{1\leq \deg B\leq n\\ B\in\mathcal{D}(\mathcal{S})}}\mu(B)\Omega(B) q^{n-\deg B}.
		\end{equation}
		Combining  \eqref{eq:equivalence_eq1} and \eqref{eq:equivalence_eq2}, we obtain
		\begin{equation}\label{eq:equivalence_eq3}
			\sideset{}{'}\sum_{\substack{1\leq \deg B\leq n\\ B\in\mathcal{D}(\mathcal{S})}}\frac{\mu(B)\Omega(B)}{q^{\deg B}}=\frac{1}{q^n}\sideset{}{'}\sum_{\deg A=n}\left(Q_\mathcal{S}^{(2)}(A)-Q_\mathcal{S}^{(1)}(A)\right).
		\end{equation}
		Thus, if we assume that as $n \to \infty$,
		\begin{align}
			\frac{1}{q^n}\sideset{}{'}\sum_{\deg A = n}Q_\mathcal{S}^{(2)}(A)\to\delta(\mathcal{S}),
		\end{align}
		then by Lemma \ref{lem:QS1},  we have from \eqref{eq:equivalence_eq3} that
		\begin{align}
			\lim_{x\to\infty}\sideset{}{'}\sum_{\substack{1\leq\deg A\leq x\\A\in\mathcal{D}(\mathcal{S})}}\frac{\mu(A)\Omega(A)}{q^{\deg A}}=0.
		\end{align}
		Conversely, if the left hand side of \eqref{eq:equivalence_eq3} tends to zero as $n \to \infty$, then adding \eqref{eq:equivalence_eq3} with \eqref{eq:QS1} implies 
		\begin{align}
			\frac{1}{q^n}\sideset{}{'}\sum_{\deg A = n}Q_\mathcal{S}^{(2)}(A)\to\delta(\mathcal{S}),
		\end{align}
		as $n \to \infty$. This proves the lemma.
	\end{proof}
	
	So, by Lemma \ref{lem:equivalence}, to prove Theorem \ref{thm:mainthm AJ}, it suffices to only show that
	\begin{align}
		\frac{1}{q^n}\sideset{}{'}\sum_{\deg A = n}Q_\mathcal{S}^{(2)}(A)\to\delta(\mathcal{S}),
	\end{align}
	as $n \to \infty$. We first prove this result for $\mathcal{S}$ being $\mathcal{P}$, the set of all primes. By definition, we have 
	\begin{align}
		\sideset{}{'}\sum_{\deg A=n}Q_\mathcal{P}^{(2)}(A)=\sideset{}{'}\sum_{\deg A=n}\sideset{}{'}\sum_{\substack{\text{prime }P\mid A\\ \deg P=\Delta_2(A)}}1.
	\end{align}
	Since $\Delta_2(A)<\frac{\deg A}{2}$, by changing the order of summation, we have
	\begin{align}\label{imp1}
		\sideset{}{'}\sum_{\deg A=n}Q_\mathcal{P}^{(2)}(A)=\sideset{}{'}\sum_{\substack{\text{prime }P\\ \deg P <\frac{n}{2}}}\hspace{2mm} \sideset{}{'}\sum_{\substack{P \mid A, \hspace{1mm}\deg A=n\\ \Delta_2(A)=\deg P}}1.
	\end{align}
	For a fixed prime $P$, let $A$ be a monic such that $\Delta_2(A)=\deg P$. Then $B=A/P$ will be a monic polynomial such that $\deg B = \deg A -\deg P$ and $\Delta_2(B)\leq \deg P$. Hence, the inner sum in \eqref{imp1} can be bounded by
	\begin{equation}\label{imp2}
		\sideset{}{'}\sum_{\substack{P \mid A, \hspace{1mm}\deg A=n\\ \Delta_2(A)=\deg P}}1 \leq \sideset{}{'}\sum_{\substack{\deg B = n-\deg P\\ \Delta_2(B)\leq \deg P}}1.
	\end{equation}
	This motivates us to find a suitable estimate for the quantity $\Psi_2(n,m)=\#S_2(n,m)$, where
	\begin{align}\label{def:S2}
		S_2(n,m)=\{\text{monic }A \in \mathbb{F}_q[T]: \deg A = n, \Delta_2(A)\leq m\}.
	\end{align}

	We are now ready to obtain an asymptotic estimate for $$\sideset{}{'}\sum_{\deg A =n}Q_\mathcal{P}^{(2)}(A).$$

	\begin{lemma}\label{QP}
		As $n \rightarrow \infty$, 
		\begin{equation}
			\frac{1}{q^n}\sideset{}{'}\sum_{\deg A=n}Q_\mathcal{P}^{(2)}(A) \to 1.
		\end{equation}
	\end{lemma}
	
	\begin{proof}
		From \eqref{imp1}, we have
		\begin{equation}
			\sideset{}{'}\sum_{\deg A=n}Q_\mathcal{P}^{(2)}(A)=\sideset{}{'}\sum_{\substack{\text{prime }P\\ \deg P<\frac{n}{2}}}\hspace{2mm}\sideset{}{'}\sum_{\substack{P \mid A, \hspace{1mm}\deg A=n\\ \Delta_2(A)=\deg P}}1.
		\end{equation}
		In order to exploit Theorem \ref{thm:Psi2}, we choose $m:=m(n)\geq 0$ such that $m \rightarrow \infty$ as $n \rightarrow \infty$ and $m=o(n)$. We split the sum over $P$ at $\deg P = m$ into $\mathcal{H}$ and $\mathcal{T}$, where
		\begin{align}
			\mathcal{H}:=\sideset{}{'}\sum_{\substack{\text{prime }P\\ \deg P\leq m}}\hspace{2mm}\sideset{}{'}\sum_{\substack{P \mid A, \hspace{1mm}\deg A=n\\ \Delta_2(A)=\deg P}}1,
		\end{align}
		and
		\begin{align}\label{tail def}
			\mathcal{T}:=\sideset{}{'}\sum_{\substack{\text{prime }P\\ m < \deg P < \frac{n}{2}}}\hspace{2mm}\sideset{}{'}\sum_{\substack{P \mid A, \hspace{1mm}\deg A=n\\ \Delta_2(A)=\deg P}}1.
		\end{align}
		We first estimate the quantity $\mathcal{H}$. As done in \eqref{imp2}, we have the bound
		\begin{align}
			\mathcal{H} \leq \sideset{}{'}\sum_{\substack{\text{prime }P\\ \deg P\leq m}}\Psi_2(n-\deg P,\deg P)= \sum_{d\leq m} \Psi_2(n-d,d)\sideset{}{'}\sum_{\substack{\text{prime }P\\ \deg P=d}}1.
		\end{align}
		Note that for all $d\leq m$, we have $d=o(n-d)$ since $m=o(n)$. So, by Theorem \ref{thm:Psi2} and the prime number theorem, we have that
		\begin{align}
			\mathcal{H} \ll \sum_{d\leq m}\left(\frac{q^{n-d}}{(n-d)^\epsilon}+q^{n-d}\cdot \frac{d}{n-d}\right)\cdot \frac{q^d}{d},
		\end{align}
		for any $0<\epsilon<1$. Since $d\leq m\leq \frac{n}{2}$, we have $\frac{1}{(n-d)^\epsilon} \leq \frac{2^\epsilon}{n^\epsilon}$ and $\frac{1}{n-d} \leq \frac{2}{n}$. Hence, we have
		\begin{equation}
			\mathcal{H} \ll q^n \frac{\log m}{n^\epsilon}+q^n \frac{m}{n}.
		\end{equation}
		Since $m=o(n)$, we obtain
		\begin{align}\label{head}
			\mathcal{H}=o(q^n).
		\end{align} 
		To estimate $\mathcal{T}$, we  write
		\begin{equation}
			\begin{split}
				\mathcal{T} = \sideset{}{'}\sum_{\substack{\text{prime }P\\ m < \deg P<\frac{n}{2}}}\sideset{}{'}\sum_{\substack{P \mid A, \hspace{1mm}\deg A=n\\ \Delta_2(A)=\deg P}}1 &= \sideset{}{'}\sum_{\deg A=n}\sideset{}{'}\sum_{\substack{\text{prime }P\mid A\\ m<\deg P=\Delta_2(A)<\frac{n}{2}}}1,
			\end{split}
		\end{equation}
		upon changing the order of summation.
		Since there are exactly $q^n$ monics of degree $n$, we can write
		\begin{equation}
			\mathcal{T}=q^n + \sideset{}{'}\sum_{\deg A=n}\left(\left(\sideset{}{'}\sum_{\substack{\text{prime }P\mid A\\ m<\deg P=\Delta_2(A)<\frac{n}{2}}}1\right)\hspace{3mm}-\hspace{3mm}1\right).
		\end{equation}
		Note that we have $A\in S_2(n,m)$ (see \eqref{def:S2}) if and only if
		\begin{align}
			\sideset{}{'}\sum_{\substack{\text{prime }P\mid A\\ m<\deg P=\Delta_2(A)<\frac{n}{2}}}1 = 0.
		\end{align}
		Hence, we obtain
		\begin{align}
			\mathcal{T} - q^n \gg -\Psi_2(n,m).
		\end{align}
		By Theorem \ref{thm:Psi2}, 
		\begin{align}\label{tail-lower}
			\mathcal{T} - q^n \gg_\epsilon -\frac{q^n}{n^\epsilon}- q^n\frac{m}{n}.
		\end{align}
		We now show that 
		\begin{align}
			\mathcal{T} - q^n \ll \frac{q^n}{m}.
		\end{align}
		Note that if $Q_\mathcal{P}^{(2)}(A)=1$, then 
		\begin{equation}
			\left(\sideset{}{'}\sum_{\substack{\text{prime }P\mid A\\ m<\deg P=\Delta_2(A)<\frac{n}{2}}}1\right) \hspace{3mm} - \hspace{3mm} 1 \hspace{3mm} \leq \hspace{3mm} 0.
		\end{equation}
		So, we can write 
		\begin{equation}
			\mathcal{T} - q^n \leq \sideset{}{'}\sum_{\substack{\deg A=n\\ Q_{\mathcal{P}}^{(2)}(A)\geq 2}}\hspace{2mm}\sideset{}{'}\sum_{\substack{\text{prime }P\mid A\\ m<\deg P=\Delta_2(A)<\frac{n}{2}}}1.
		\end{equation}
		Interchanging the order of summation again yields
		\begin{equation}
			\mathcal{T} - q^n \leq \sideset{}{'}\sum_{\substack{P \text{ prime }\\ m<\deg P<\frac{n}{2}}} \#\{\text{monic }A\in \mathbb{F}_q[T]:\deg A=n, P\mid A, Q_\mathcal{P}^{(2)}(A)\geq 2, \Delta_2(A)=\deg P\}.
		\end{equation}
		We write such $A$ as $A=PP'B$, where $P,P'$ are primes of degree $\Delta_2(A)$. As a result,
		\begin{equation}
			\begin{split}
				\mathcal{T} - \hspace{2mm}q^n &\leq \sideset{}{'}\sum_{m<d<\frac{n}{2}}\sideset{}{'}\sum_{\substack{\text{prime }P,P'\\ \deg P'=\deg P=d}} 1 \sideset{}{'}\sum_{\substack{\deg B =n-2d}}1\\
				&\ll \sideset{}{'}\sum_{m<d<\frac{n}{2}} \frac{q^{2d}}{d^2}q^{n-2d}\ll \frac{q^n}{m}.
			\end{split}
		\end{equation}
		Thus, we obtain
		\begin{align}\label{tail-upper}
			\mathcal{T} - q^n \ll \frac{q^n}{m}.
		\end{align}
		Combining \eqref{tail-upper} and \eqref{tail-lower}, and taking $m=m(n)$ such that $m=o(n)$ and $m \to \infty$ as $n\to \infty$, we obtain
		\begin{align}\label{tail}
			\mathcal{T} = q^n + o(q^n).
		\end{align}
		Hence, combining  \eqref{head} and \eqref{tail}, we obtain that as $n\to\infty$,
		\begin{equation}
			\frac{1}{q^n}\sideset{}{'}\sum_{\deg A=n}Q_\mathcal{P}^{(2)}(A) \to 1.
		\end{equation}
	\end{proof}
	The result in Lemma \ref{QP} is instrumental in obtaining an asymptotic estimate for $$\sideset{}{'}\sum_{\deg A =n}Q_\mathcal{S}^{(2)}(A),$$
	where $\mathcal{S}$ is any arbitrary subset of $\mathcal{P}$. We derive this in the following lemma.
	
	\begin{lemma}\label{QS}
		As $n \rightarrow \infty$, we have
		\begin{equation}
			\frac{1}{q^n}\sideset{}{'}\sum_{\deg A=n}Q_\mathcal{S}^{(2)}(A) \to \delta(\mathcal{S}).
		\end{equation}
	\end{lemma}
	
	\begin{proof}
		
		Similar to \eqref{imp1}, we can write
		\begin{equation}
			\sideset{}{'}\sum_{\deg A=n}Q_\mathcal{S}^{(2)}(A)=\sideset{}{'}\sum_{\substack{P\in \mathcal{S}\\ \deg P <\frac{n}{2}}}\hspace{2mm} \sideset{}{'}\sum_{\substack{P \mid A, \hspace{1mm}\deg A=n\\ \Delta_2(A)=\deg P}}1.
		\end{equation}
		Let $m=m(n)$ to be chosen later such that $m \rightarrow \infty$ as $n \rightarrow \infty$ and $m=o(n)$. We then split the double sum into $\mathcal{H}'$ and $\mathcal{T}'$ where
		\begin{align}
			\mathcal{H}':=\sideset{}{'}\sum_{\substack{P\in \mathcal{S}\\ \deg P \leq m}}\hspace{2mm} \sideset{}{'}\sum_{\substack{P \mid A, \hspace{1mm}\deg A=n\\ \Delta_2(A)=\deg P}}1,
		\end{align}
		and
		\begin{align}
			\mathcal{T}':=\sideset{}{'}\sum_{\substack{P\in \mathcal{S}\\ m< \deg P < \frac{n}{2}}}\hspace{2mm} \sideset{}{'}\sum_{\substack{P \mid A, \hspace{1mm}\deg A=n\\ \Delta_2(A)=\deg P}}1.
		\end{align}
		Since $\mathcal{S}\subset\mathcal{P}$, it follows by definition that $\mathcal{H}'\leq \mathcal{H}$. So, by \eqref{head}, we have
		\begin{align}\label{head'}
			\mathcal{H}'= o(q^n).
		\end{align}
		It remains to estimate $\mathcal{T}'$. Note that 
		\begin{align}
			\mathcal{T}' = \# \{(P, A): P\in \mathcal{S}, m < \deg P < n/2, P \mid A, \deg A=n, \Delta_2(A)=\deg P\}.
		\end{align}
		Fix an arbitrary total ordering $\preceq$ on the primes in $\mathbb{F}_q[T]$. For every pair $(A,P)$ occurring in $\mathcal T'$, let $Q_1,\cdots,Q_r$ be all the prime factors of $A$ of degree $\Delta_1(A)$, listed with multiplicity and in nondecreasing order
		$$
		Q_1\preceq Q_2\preceq\cdots\preceq Q_r.
		$$
		Then
		$
		A=Q_1Q_2\cdots Q_rPB
		$,
		where $\Delta_1(B)\le \Delta_2(A)=\deg P<\Delta_1(A)$. Conversely, every admissible tuple
		$(Q_1,\ldots,Q_r,P,B)$ satisfying the above conditions determines a
		unique pair $(A,P)$, and every pair $(A,P)$ counted by $\mathcal{T}'$
		arises uniquely in this way. For ease of notation, let 
		\begin{align}\label{def:tuple}
			N_r(k)=\#\{(Q_1,\cdots, Q_r): Q_i \text{ primes, not necessarily distinct}, Q_1\preceq Q_2\preceq\cdots\preceq Q_r, \deg Q_i = k\}.
		\end{align} 
		Thus, we can write
		\begin{equation}
			\begin{split}
				\mathcal{T}' 
				& = \sum_{r=1}^n
				\sum_{1 \leq k \leq \lfloor n/r \rfloor}
				N_r(k)\sideset{}{'}\sum_{\substack{P\in \mathcal{S}\\ m < \deg P<\frac{n}{2}\\ \deg P <k}}\sideset{}{'}\sum_{\substack{\deg B = n-rk-\deg P \\ \Delta_1(B)\leq\deg P}}1.
			\end{split}
		\end{equation}
		For easier estimation, we write $\mathcal{T}' = \mathcal{T}'_1 + \mathcal{T}'_2$, where
		\begin{align}
			\mathcal{T}'_1 := \sum_{r=1}^n
			\sum_{1\leq k \leq \min\left\{\lfloor \frac{n}{2} \rfloor,\lfloor \frac{n}{r} \rfloor\right\}}
			N_r(k) \sideset{}{'}\sum_{\substack{P\in \mathcal{S}\\ m < \deg P<k}}\sideset{}{'}\sum_{\substack{\deg B = n-rk-\deg P \\ \Delta_1(B)\leq\deg P}}1
		\end{align}
		and
		\begin{align}
			\mathcal{T}'_2 := \sum_{r=1}^n
			\sum_{\lfloor n/2 \rfloor < k \leq \lfloor n/r \rfloor}
			N_r(k)\sideset{}{'}\sum_{\substack{P\in \mathcal{S}\\ m < \deg P<\frac{n}{2}}}\sideset{}{'}\sum_{\substack{\deg B = n-rk-\deg P \\ \Delta_1(B)\leq\deg P}}1.
		\end{align}
		Recall from \eqref{tail def} that 
		\begin{equation}
			\mathcal{T} =\sideset{}{'}\sum_{\substack{\text{prime }P\\ m< \deg P < \frac{n}{2}}}\hspace{2mm} \sideset{}{'}\sum_{\substack{P \mid A, \hspace{1mm}\deg A=n\\ \Delta_2(A)=\deg P}}1.
		\end{equation}
		Similar to $\mathcal{T}'$, we write $\mathcal{T}=\mathcal{T}_1+\mathcal{T}_2$, where
		\begin{align}
			\mathcal{T}_1 := \sum_{r=1}^n
			\sum_{1\leq k \leq \min\left\{\lfloor \frac{n}{2} \rfloor,\lfloor \frac{n}{r} \rfloor\right\}}
			N_r(k)\sideset{}{'}\sum_{\substack{\text{prime }P\\ m < \deg P<k}}\sideset{}{'}\sum_{\substack{\deg B = n-rk-\deg P \\ \Delta_1(B)\leq\deg P}}1
		\end{align}
		and
		\begin{align}
			\mathcal{T}_2 :=\sum_{r=1}^n
			\sum_{\lfloor n/2 \rfloor < k \leq \lfloor n/r \rfloor}
			N_r(k)\sideset{}{'}\sum_{\substack{\text{prime }P\\ m < \deg P<\frac{n}{2}}}\sideset{}{'}\sum_{\substack{\deg B = n-rk-\deg P \\ \Delta_1(B)\leq\deg P}}1.
		\end{align}
		To estimate $\mathcal{T}'$, we first estimate the difference
		\begin{equation}
			\mathcal{T}' - \delta(\mathcal{S})\mathcal{T} = \left(\mathcal{T}'_1 - \delta(\mathcal{S})\mathcal{T}_1\right)+\left(\mathcal{T}'_2 - \delta(\mathcal{S})\mathcal{T}_2\right).
		\end{equation}
		We write 
		\begin{equation}
			\begin{split}
				\mathcal{T}'_1 - \delta(\mathcal{S})\mathcal{T}_1 &= \sum_{r=1}^n
				\sum_{1\leq k \leq \min\left\{\lfloor \frac{n}{2} \rfloor,\lfloor \frac{n}{r} \rfloor\right\}}
				N_r(k)\left(\sideset{}{'}\sum_{\substack{P\in \mathcal{S}\\ m < \deg P<k}}\sideset{}{'}\sum_{\substack{\deg B = n-rk-\deg P \\ \Delta_1(B)\leq\deg P}}1-\delta(\mathcal{S})\sideset{}{'}\sum_{\substack{\text{prime }P\\ m < \deg P<k}}\sideset{}{'}\sum_{\substack{\deg B = n-rk-\deg P \\ \Delta_1(B)\leq\deg P}}1\right)\\
				&= \sum_{r=1}^n
				\sum_{1\leq k \leq \min\left\{\lfloor \frac{n}{2} \rfloor,\lfloor \frac{n}{r} \rfloor\right\}}
				N_r(k)
				\sum_{m < d<k} \hspace{2mm} \sideset{}{'}\sum_{\substack{\deg B = n-rk-d \\ \Delta_1(B)\leq d}}(\pi_\mathcal{S}(d)-\delta(\mathcal{S})\pi_\mathcal{P}(d)),
			\end{split}
		\end{equation}
		where $\pi_\mathcal{S}(d)$ denotes the cardinality of the set $\{P\in\mathcal{S}:P \,\, prime, \deg P = d\}$. We further define the following two suprema, as done in \cite{duan-wang-yi}:
		\begin{align}
			e_\mathcal{S}(d) :=\sup_{e\leq d}|\pi_\mathcal{S}(e)-\delta(\mathcal{S})\pi_\mathcal{P}(e)|, \quad
			v_\mathcal{S}(d) :=\sup_{e\geq d}\frac{e_\mathcal{S}(e)}{q^e}.
		\end{align}
		We note that $v_\mathcal{S}(d)$ is a nonincreasing function. As a result, we have 
		\begin{align}
			|\mathcal{T}'_1 - \delta(\mathcal{S})\mathcal{T}_1| &\ll \sum_{r=1}^n
			\sum_{1\leq k \leq \min\left\{\lfloor \frac{n}{2} \rfloor,\lfloor \frac{n}{r} \rfloor\right\}}
			N_r(k)
			\sum_{m<d<k} e_\mathcal{S}(d) q^{n-rk-d}\\
			&= q^n \sum_{r=1}^n
			\sum_{1\leq k \leq \min\left\{\lfloor \frac{n}{2} \rfloor,\lfloor \frac{n}{r} \rfloor\right\}}
			N_r(k)q^{-rk}
			\sum_{m<d<k} e_\mathcal{S}(d) q^{-d}\\
			&\ll q^n\sum_{r=1}^n
			\sum_{1\leq k \leq \min\left\{\lfloor \frac{n}{2} \rfloor,\lfloor \frac{n}{r} \rfloor\right\}}
			N_r(k)q^{-rk}kv_\mathcal{S}(m). \label{T1-wait}
		\end{align}
		We have 
		\begin{align}
			\sum_{r=1}^n
			\sum_{1\leq k\leq
				\min\left\{
				\left\lfloor\frac{n}{2}\right\rfloor,
				\left\lfloor\frac{n}{r}\right\rfloor
				\right\}}
			N_r(k)kq^{-rk}
			&=
			\sum_{k\leq n/2}
			N_1(k)kq^{-k}
			+
			\sum_{r=2}^n
			\sum_{k\leq n/r}
			N_r(k)kq^{-rk}
			\nonumber\\
			&=
			\sum_{k\leq n/2}
			\pi_{\mathcal{P}}(k)kq^{-k}
			+
			\sum_{r=2}^n
			\sum_{k\leq n/r}
			N_r(k)kq^{-rk}.
			\label{eq:Nr-bound-first}
		\end{align}
		Since
		\begin{align}
			N_r(k)
			\leq
			\pi_{\mathcal{P}}(k)^r
			\leq
			\frac{q^{rk}}{k^r},
		\end{align}
		we obtain
		\begin{align}
			\sum_{r=1}^n
			\sum_{1\leq k\leq
				\min\left\{
				\left\lfloor\frac{n}{2}\right\rfloor,
				\left\lfloor\frac{n}{r}\right\rfloor
				\right\}}
			N_r(k)kq^{-rk}
			&\leq
			\sum_{k\leq n/2}1
			+
			\sum_{k\leq n/2}
			\sum_{r=2}^{\lfloor n/k\rfloor}
			k^{1-r}
			\nonumber\\
			&\leq
			\frac{n}{2}
			+
			\sum_{r=2}^{n}1
			+
			\sum_{k=2}^{\lfloor n/2\rfloor}
			\sum_{r=2}^{\infty}k^{1-r}
			\nonumber\\
			&=
			\frac{n}{2}
			+
			n-1
			+
			\sum_{k=2}^{\lfloor n/2\rfloor}
			\frac{1}{k-1}
			\nonumber\\
			&\ll n+\log n
			\ll n.
			\label{eq:Nr-bound}
		\end{align}
		Plugging this in \eqref{T1-wait}, we obtain
		\begin{align}
			|\mathcal{T}'_1 - \delta(\mathcal{S})\mathcal{T}_1| \ll
			q^n nv_\mathcal{S}(m).\label{T1}
		\end{align} 
		Now in order to estimate $|\mathcal{T}'_2 - \delta(\mathcal{S})\mathcal{T}_2|$, we first observe that in this case, the sum over $r$ survives only when $r=1$. So, as done for $|\mathcal{T}'_1 - \delta(\mathcal{S})\mathcal{T}_1|$, we have that
		\begin{align}
			|\mathcal{T}'_2 - \delta(\mathcal{S})\mathcal{T}_2| &\ll \sideset{}{'}\sum_{\substack{\text{prime }Q\\ \frac{n}{2}\leq \deg Q = k \leq n}}\hspace{1mm}
			\sum_{m<d<k} e_\mathcal{S}(d) q^{n-k-d}\\
			&= q^n \sideset{}{'}\sum_{\substack{\text{prime }Q\\ \frac{n}{2}\leq \deg Q = k \leq n}}q^{-k}
			\sum_{m<d<k} e_\mathcal{S}(d) q^{-d}\\
			&\ll q^n \sideset{}{'}\sum_{\substack{\text{prime }Q\\ \frac{n}{2}\leq \deg Q = k \leq n}} q^{-k}kv_\mathcal{S}(m)\\
			&\ll q^n nv_\mathcal{S}(m).\label{T2}
		\end{align}
		Adding \eqref{T1} and \eqref{T2}, we get 
		\begin{equation}
			\mathcal{T}'
			=\delta(\mathcal{S})\mathcal{T}
			+
			O(q^n nv_\mathcal{S}(m))
			.
		\end{equation}
		Since  $\mathcal{H}'=o(q^n)$ by \eqref{head'}, we obtain
		\begin{equation}
			\sideset{}{'}\sum_{\deg A=n}Q_\mathcal{S}^{(2)}(A) 
			= 
			\mathcal{H}'+\mathcal{T}' 
			= 
			\delta(\mathcal{S})\mathcal{T}+ 
			O(q^n nv_\mathcal{S}(m))
			+o(q^n).
		\end{equation}
		By \eqref{tail}, we have $\mathcal{T}=q^n + o(q^n)$. Thus, we have
		\begin{equation}
			\sideset{}{'}\sum_{\deg A=n}Q_\mathcal{S}^{(2)}(A)
			= \delta(\mathcal{S})q^n
			+O(q^n nv_\mathcal{S}(m)) + o(q^n).
		\end{equation}
		So, it only remains to show that $q^n nv_\mathcal{S}(m) = o(q^n)$ or equivalently, $\lim_{n\to\infty} nv_\mathcal{S}(m)=0$. 
		Here, we invoke Lemma \ref{lem:vs} and its following remark. Let $f(n)=\pi_\mathcal{S}(n)$ and $g(n)=\pi_\mathcal{P}(n)$ as defined above. Then, by Lemma \ref{lem:vs}(a), 
		\begin{equation}
			\lim_{n\to\infty}\frac{e_\mathcal{S}(n)}{\pi_\mathcal{P}(n)}=0.
		\end{equation}
		By Lemma \ref{lem:PNT in ff}, this implies $\frac{n e_{\mathcal S}(n)}{q^n} \to 0$ as $n\to \infty$. As a result,
		\begin{align}
			n v_{\mathcal S}(n)
			\le
			\sup_{j\ge n}
			\frac{j e_{\mathcal S}(j)}{q^j}
			\to 0,
		\end{align}
		as $n \to \infty$. If $v_{\mathcal S}$ is eventually zero, choose any sequence $m(n)\to\infty$ with $m(n)=o(n)$. Otherwise, $v_{\mathcal S}(n)>0$ for all $n$, and  then $v_\mathcal{S}(n)$ satisfies the hypothesis for Lemma \ref{lem:vs}(b). So, by Lemma \ref{lem:vs}(b), choosing $h(n)=\sqrt{n},$ there exists $m(n)$ such that $m(n)=o(n)$, $\lim_{n\to\infty}m(n)=\infty$ and $\lim_{n\to\infty}nv_\mathcal{S}(m(n)) = 0$. Hence, we have that as $n\to \infty$,
		\begin{equation}
			\frac{1}{q^n}\sideset{}{'}\sum_{\deg A=n}Q_\mathcal{S}^{(2)}(A) \to \delta(\mathcal{S}).
		\end{equation}
	\end{proof}

	\section
	{Proof of Corollary \ref{cor:to main thm in ff}}
	
	Let $P_k$ be a prime such that $\deg P_k = k$ and define $\mathcal{S}_k=\{P_k\}$. Then
	\begin{equation}
		\mathcal{D}(\mathcal{S}_k)=\{P_k^{\alpha}B : B=1 \text{ or } \delta_1(B)\geq k+1, \alpha \geq 1\}.
	\end{equation}
	So, the sum 
	\begin{align}
		\lim_{x\to\infty}\sideset{}{'}\sum_{\substack{1\leq\deg A\leq x\\A\in\mathcal{D}(\mathcal{S}_k)}}\frac{\mu(A)\Omega(A)}{q^{\deg A}} 
		&=
		-\frac{1}{q^k} \left( 1+ \lim_{x\to\infty}
		\sideset{}{'}
		\sum_{\substack{1  \leq \deg B \leq x-k  \\  \delta_1(B)\geq k+1}}
		\frac{\mu(B)\Big(\Omega(B)+1\Big)}{q^{\deg B}} \right)
		\\
		&=
		-\frac{1}{q^k} \left(1+
		\lim_{x\to\infty}
		\left(
		\sideset{}{'}
		\sum_{\substack{1  \leq \deg B \leq x-k  \\  \delta_1(B)\geq k+1}}
		\frac{\mu(B)\Omega(B)}{q^{\deg B}}
		+
		\sideset{}{'}
		\sum_{\substack{1  \leq \deg B \leq x-k  \\  \delta_1(B)\geq k+1}}
		\frac{\mu(B)}{q^{\deg B}}\right)\right). \label{mids}
	\end{align}
	Mertens' estimate \eqref{merten} implies that
	\begin{align}
		\prod_{\text{prime }P}
		\left(1-\frac{1}{q^{\deg P}}\right)
		=0.
	\end{align}
	For a fixed positive integer $k$, define
	\begin{align}
		G_k(u)
		:=
		\sideset{}{'}\sum_{\substack{A=1\text{ or }\delta_1(A)\geq k+1}}
		\mu(A)u^{\deg A}.
	\end{align}
	Then, as a formal power series,
	\begin{align}
		G_k(u)
		=
		\prod_{\substack{\text{prime }P\\ \deg P\geq k+1}}
		\left(1-u^{\deg P}\right)
		=
		\frac{\displaystyle\prod_{\text{prime }P}
			\left(1-u^{\deg P}\right)}
		{\displaystyle\prod_{\substack{\text{prime }P\\ \deg P\leq k}}
			\left(1-u^{\deg P}\right)}
		=
		\frac{1-qu}
		{\displaystyle\prod_{\substack{\text{prime }P\\ \deg P\leq k}}
			\left(1-u^{\deg P}\right)}.
	\end{align}
	Since the denominator has no zero for $|u|<1$, the Taylor series of $G_k(u)$ about $u=0$ converges at $u=1/q$. Hence, we obtain
	\begin{align}
		\lim_{x\to\infty}
		\sum_{\substack{\deg A\leq x\\
				A=1\text{ or }\delta_1(A)\geq k+1}}
		\frac{\mu(A)}{q^{\deg A}}
		=
		G_k\left(\frac{1}{q}\right)
		=0.
	\end{align}
	Since the term $A=1$ contributes $1$, it follows that
	\begin{align}\label{mu tail}
		\lim_{x\to\infty}
		\sum_{\substack{1\leq \deg A\leq x\\
				\delta_1(A)\geq k+1}}
		\frac{\mu(A)}{q^{\deg A}}
		=-1.
	\end{align}
	Using \eqref{mu tail} and Theorem \ref{thm:mainthm AJ}, we obtain from \eqref{mids} that  for $k\geq 1$, 
	\begin{equation}
		\lim_{x\to\infty}\sideset{}{'}\sum_{\substack{1  \leq \deg B \leq x-k  \\  \delta_1(B)\geq k+1}}\frac{\mu(B)\Omega(B)}{q^{\deg B}}=0.
	\end{equation}
	This is equivalent to saying that
	\begin{equation}
		\lim_{x\to\infty}\sideset{}{'}\sum_{\substack{1  \leq \deg A \leq x  \\  \delta_1(A)\geq k+1}}\frac{\mu(A)\Omega(A)}{q^{\deg A}}=0.
	\end{equation}
	We already have from Theorem \ref{thm:landau_two_FF} that,
	\begin{equation}
		\lim_{x\to\infty}
		\sideset{}{'}\sum_{\substack{1\leq\deg A\leq x \\ \delta_1(A)\geq 1}}
		\frac{\mu(A)\Omega(A)}{q^{\deg A}}
		=
		\lim_{x\to\infty}
		\sideset{}{'}\sum_{1\leq\deg A\leq x}
		\frac{\mu(A)\Omega(A)}{q^{\deg A}}
		=0.
	\end{equation}
	As a result, for every $n \in \mathbb{N}$, 
	\begin{equation}
		\lim_{x\to\infty}
		\sideset{}{'}\sum_{\substack{1\leq\deg A\leq x \\ \delta_1(A)=n}}
		\frac{\mu(A)\Omega(A)}{q^{\deg A}}
		=
		\lim_{x\to\infty}
		\sideset{}{'}\sum_{\substack{1\leq\deg A\leq x\\ \delta_1(A)\geq n}}
		\frac{\mu(A)\Omega(A)}{q^{\deg A}}
		-
		\lim_{x\to\infty}
		\sideset{}{'}\sum_{\substack{1\leq\deg A\leq x \\ \delta_1(A)\geq n+1}}
		\frac{\mu(A)\Omega(A)}{q^{\deg A}}=0.
	\end{equation}
	Hence, we have the result. \qed
	
	\section{Proof of Theorem \ref{thm:global_duality}}
	
	The proof follows the same argument as done in the proof of Theorem
	\ref{thm:duality in ff}. The result is immediate when $A=0$. Moreover,
	if $k>\Omega(A)$, then both sides of \eqref{higher duality in global ff}
	vanish since $f(0)=0$. Hence, we may assume that
	$1\leq k\leq\Omega(A)$. For $k=1$, the result follows from the global
	function field analogue established by Duan, Wang, and Yi
	\cite[Lemma 3.1]{duan-wang-yi}. Assuming that
	\eqref{higher duality in global ff} holds for some integer
	$1\leq r\leq\Omega(A)-1$, Pascal's identity yields
	\begin{align}		\label{eq:global_sum_Pascal}
		&\sum_{\substack{B\in\mathfrak{D}(K,\mathcal{S}) \\ B \le A}} \mu(B)\binom{\Omega(B)-1}{r}f(\delta_1(B)) + (-1)^r Q_\mathcal{S}^{(r)}(A)f(\Delta_r(A))\\
		&= \sum_{\substack{B\in\mathfrak{D}(K,\mathcal{S}) \\ B \le A}} \mu(B) \binom{\Omega(B)}{r} f(\delta_1(B)). 
	\end{align}
	
	We now let $\Omega(A)=m$ and
	$\mathcal{N}(A)=\{d_1,d_2,\cdots,d_m\}$, where
	$d_1<d_2<\cdots<d_m$. We decompose $A$ as
	$A=A_1+\cdots+A_m$, where each effective divisor $A_i$ is the sum,
	with multiplicities, of all the prime divisors of $A$ having degree
	$d_i$. For any divisor $B\in\mathfrak{D}(K,\mathcal{S})$, let $t(B)$
	be the index such that the smallest prime factor of $B$ has degree
	$d_{t(B)}$. Since the terms with $\mu(B)=0$ do not contribute, we may
	restrict to squarefree divisors $B$. Then any such $B\leq A$ can be
	written as
	$
	B=P_t+B'$,
	where $t=t(B)$, $P_t\in\mathcal{S}$ is a prime with
	$\deg P_t=d_t$, and
	$B'\leq A_{t+1}+\cdots+A_m$. Under this setup, we follow the same
	argument after \eqref{former} to get

	\begin{align}\label{eq:penultimate}
		\sum_{\substack{B\in\mathfrak{D}(K,\mathcal{S}) \\ B \le A}} \mu(B) \binom{\Omega(B)}{r} f(\delta_1(B))= - \sum_{t=1}^m f(d_t) Q_\mathcal{S}^{(m-t+1)}(A) \sum_{B' \le A_{t+1} + \cdots + A_m} \mu(B') \binom{\Omega(B')+1}{r}.
	\end{align}
	Then, a reformulation of Corollary \ref{lem:sum muomega gen in ff} for global function fields shows that the inner sum
	\begin{align}
		\sum_{B' \le A_{t+1} + \cdots + A_m} \mu(B') \binom{\Omega(B')+1}{r} = 
		\begin{cases}
			(-1)^{r-1} & \text{ if }r=m-t+1,\\
			(-1)^r & \text{if }r=m-t,\\
			0 & \text{otherwise.}
		\end{cases}
	\end{align}
	Substituting this in \eqref{eq:penultimate} and then combining \eqref{eq:penultimate} and \eqref{eq:global_sum_Pascal}, we obtain \eqref{higher duality in global ff} after simplification. 
	\qed

	\section*{Acknowledgment}
	The authors  express their sincere gratitude to Prof. K. Alladi for his inspiring talk on "Second Order Duality Between Prime Factors and
	Primes in Arithmetic Progressions" and for the discussion session during his visit to IIT Gandhinagar, which motivated the initiation of this project. The authors are also deeply grateful to Prof. Akshaa Vatwani for suggesting the problem, engaging in several  discussions, and carefully reading the article, thereby improving  the quality of the paper. They further thank Sroyon Sengupta for insightful discussions which helped in deriving a bound for $\Psi_2(n,m)$. The second author would like to thank IMA Bhubaneswar where a part of the research has been conducted.
	\bibliographystyle{abbrv}
	\bibliography{alladi}

\end{document}